\documentclass[11pt,british,reqno]{amsart}

\usepackage[T1]{fontenc}

\usepackage{babel,eucal,url,amssymb,enumerate,booktabs,amscd,graphics}

\theoremstyle{plain}
\newtheorem{lemma}{Lemma}[section]

\newtheorem{proposition}[lemma]{Proposition}
\newtheorem{corollary}[lemma]{Corollary}
\newtheorem{theorem}[lemma]{Theorem}
\newtheorem{remark}[lemma]{Remark}

\newcommand{\Lie}[1]{\operatorname{\textsl{#1}}}
\newcommand{\lie}[1]{\operatorname{\mathfrak{#1}}}

\newcommand{\SO}{\Lie{SO}}

\newcommand{\Sp}{\Lie{Sp}}
\newcommand{\asp}{\lie{sp}}
\newcommand{\Spin}{\Lie{Spin}}
\newcommand{\SU}{\Lie{SU}}

\newcommand{\Un}{\Lie{U}}

\newcommand{\Gtwo}{\ifmmode{{\rm G}_2}\else{${\rm G}_2$}\fi}

 \newcommand{\cyclic}{\mathop{\kern0.9ex{{+}\kern-2.2ex\raise-.28ex\hbox{\Large\hbox
 {$\circlearrowright$}}}}}

\def\sideremark#1{\ifvmode\leavevmode\fi\vadjust{\vbox to0pt{\vss
 \hbox to 0pt{\hskip\hsize\hskip1em
 \vbox{\hsize2.5cm\tiny\raggedright\pretolerance10000
 \noindent #1\hfill}\hss}\vbox to8pt{\vfil}\vss}}}%

\input xy
\xyoption{all}

\newfont{\eusm}{eusm10 scaled \magstep1}
\newfont{\eusmiii}{eusm10 scaled \magstep3}

\title[Sasakian structures on tangent sphere bundles]{Sasakian structures on tangent sphere bundles of compact rank-one symmetric spaces}

\author[J.~C.~Gonz{\'a}lez-D{\'a}vila]{J.~C.~Gonz{\'a}lez-D{\'a}vila}
\address{Departamento de Matem\'aticas, Estad\'istica e Investigaci\'on
Ope\-ra\-tiva, University of La Laguna, 38200 La Laguna, Tenerife, Spain.}
\email{jcgonza@ull.es}

\thanks{Research partially supported by the AEI (Spain) and FEDER project PID2019-105019GB-C21}

\begin{document}

\maketitle

\begin{abstract} A positive answer is given to the existence of Sasakian structures on the tangent sphere bundle of some Riemannian manifold whose sectional curvature is not cons\-tant. Among other results, it is proved that the tangent sphere bundle $T_{r}(G/K),$ for any $r> 0,$ of a compact rank-one symmetric space $G/K,$ not necessarily of constant sectional curvature, admits a unique K-contact structure whose characteristic vector field is the stan\-dard field of $T(G/K)$. Such a structure is in fact Sasakian and it can be expressed as an induced structure from an almost Hermitian structure on the punctured tangent bundle $T(G/K)\setminus \{\mbox{zero section}\}.$
\end{abstract}

\section{Introduction}
The tangent sphere bundle of radius $r>0$ over a Riemannian manifold $(M, g)$ is the hypersurface $T_{r}M = \{u\in TM\colon g(u,u)= r^{2}\}$ of the tangent bundle $TM.$ There is an extensive bibliography on the Riemannian geometry of $T_{r}M$ equipped with the induced metric $\widetilde{g}^{S}$ from the Sasaki metric $g^{S},$ specially about the unit tangent sphere bundle $T_{1}M.$ See E. Boeckx and L. Vanhecke \cite{BV, BV1} and G. Calvaruso \cite{Ca} for surveys of $T_{1}M$ and O. Kowalski and M. Sekizawa \cite{KS, KS1, KS2} of $T_{r}M,$ for an arbitrary radius $r>0.$ The induced metric on $T_{1}M$ from a $g$-natural metric, which generalizes the Sasaki and the Cheeger-Gromoll metric, has been treated by K. M. T. Abbassi, G. Calvaruso and M. Sarih (see, for example, \cite{AC, AC1, AS}).

With respect to the {\em standard almost complex structure} $J,$ the Sasaki metric on $TM$ is Hermitian, in fact it is almost K\"ahler \cite{TO}. Then $T_{r}M,$ as hypersurface of $(TM,J,g^{S}),$ inherits an almost contact metric structure $(\varphi,\xi,\eta, \widetilde{g}^{S})$ and the unit vector field $\xi = -JN,$ where $N$ is the outward normal unit vector field to the singular foliation $\{T_{r}M\}_{r\geq 0},$  is defined on $TM\setminus \{\mbox{\rm zero section}\}.$ We say that $\xi$ is the {\em standard vector field} of $TM.$

On $T_{1}M,$ the pair $(\frac{1}{2}\eta,\frac{1}{4}\widetilde{g}^{S})$ becomes a contact metric structure and Tashiro \cite{Tash} proved it is $K$-contact if and only if the base manifold $(M,g)$ has constant sectional curvature 1. Moreover, in this case, $T_{1}M$ is Sasakian. For $g$-natural metrics, $(M,g)$ must be of constant sectional curvature \cite[Theorem 2]{AC} and then, Sasakian too. Sasakian manifolds can be considered in many respects as the class analogous to that of the K\"ahler manifolds for the odd dimensional case. Principal circle bundles over K\"ahler manifolds, together with the standard structure on ${\mathbb R}^{2n+1},$ are the best known examples (see, for example, \cite[Ch. 6]{Bl}).  

 The main objective of this paper is to find Sasakian structures on tangent sphere bundles $T_{r}M,$ for any radius $r>0,$ over Riemannian manifolds $(M, g)$ with non-constant sectional curvature. As far as the author knows, there are no examples satisfying these conditions in the literature.
 
 The most natural Riemannian manifolds $(M, g)$ in this search should be complex or quaternionic projective spaces, or even the Cayley plane since the tangent sphere bundle of any compact rank-one symmetric space $G/K$ is, among other properties, $G$-homogeneous (see \cite{BPV, MT}) and its standard vector field $\xi$ is $G$-invariant. In order to achieve our objective, it is necessary to go further and carry out a detailed study of compact rank-one symmetric spaces and their tangent bundles and tangent sphere bundles. Such a study is included in Sections $3$ and $4,$ while in Section $2$ some basic concepts about almost contact metric structures and homogeneous manifolds are exposed.
 
 Given a compact rank-one symmetric space $G/K,$ we set a Cartan decomposition ${\mathfrak g} = {\mathfrak m}\oplus {\mathfrak k}$ and a Cartan subspace ${\mathfrak a}={\mathbb R}\{X\}$ of ${\mathfrak m}.$ Then the punctured tangent bundle $D^{+}(G/K):= T(G/K)\setminus \{\mbox{zero section}\}$ can be identified with $G/H\times {\mathbb R}^{+}$ and $T_{r}(G/K),$ for each $r>0,$ with the quotient $G/H$ via $G$-equivariant diffeomorphisms, where $H$ is the closed subgroup $H = \{k\in K\colon {\rm Ad}_{k} X=X\}$ of $K$  (Proposition \ref{phomo}). 

 The homogeneity of $T_{r}(G/K)$ is discussed in Section $4,$ where the quotient expressions $G/H,$ for each compact rank-one symmetric space, together with the sphere-homogene\-ous fibration associated to the projection $\pi^{T}\colon T_{r}(G/K) = G/H\to G/K$ are given. As a direct consequence, a natural diffeomorphism is established between $T_{r}{\mathbb K}{\mathbf P}^{n}$ and the ${\mathbb K}$-{\em projective Stiefel manifold} $W_{2}({\mathbb K}^{n+1}),$ ${\mathbb K}$ being the field of real numbers, complex numbers or quaternions. 
 
 On the other hand,  we consider the system $\Sigma$ of {\em restricted roots} of $({\mathfrak g},{\mathfrak k},{\mathfrak a}),$ which is either $\Sigma = \{\pm \varepsilon\}$ or $\Sigma=\{\pm\varepsilon,\pm\varepsilon/2\},$ where $\varepsilon\in ({\mathfrak a}^{\mathbb C})^{*},$ together with the subspaces    \[
 {\mathfrak m}_{\lambda} = \{\xi\in {\mathfrak m}\colon {\rm ad}^{2}_{X}\xi = \lambda^{2}(X)\xi\},\quad {\mathfrak k}_{\lambda}= \{\zeta\in {\mathfrak k}\colon {\rm ad}^{2}_{X}\zeta = \lambda^{2}(X)\zeta\},\quad \lambda\in \Sigma^{+}.
 \]
Because the tangent $T_{o_{H}}(G/H)$ at the origin $o_{H} = \{H\}$ is identified with $\overline{\mathfrak m} = ({\mathfrak a}\oplus {\mathfrak m}_{\varepsilon}\oplus {\mathfrak m}_{\varepsilon/2}) \oplus ({\mathfrak k}_{\varepsilon}\oplus {\mathfrak k}_{\varepsilon/2})$ and this decomposition is ${\rm Ad}(H)$-irreducible, all the $G$-invariant metrics on $T_{r}(G/K)$ are obtained (Proposition \ref{metricB}). In particular, the induced metric $\tilde{g}^{S}$ of Sasaki metric is analyzed in this context, which allows obtaining, among other results, a version of Tashiro's Theorem for tangent sphere bundles of {\em any} radius (Corollary \ref{Tashiro}).

In the last section, Section 5, we use all the material developed in the previous sections to prove our main result.  
 \begin{theorem}\label{main} Let $G/K$ be a compact rank-one symmetric space and let $\xi$ be the standard vector field of $T(G/K).$ Then, for each smooth function $f\colon {\mathbb R}^{+}\to {\mathbb R}^{+},$ there exists a $G$-invariant almost Hermitian structure on $T(G/K)\setminus\{\mbox{zero section}\}$ such that its induced almost contact metric structure on $T_{r}(G/K),$ for all $r>0,$ is $K$-contact and $f(r)^{-1}\xi$ is the characteristic vector field.
 
  Such $K$-contact structure on $T_{r}(G/K),$ for each $r>0,$ is Sasakian and it is given by $(\kappa^{-1}\xi,\tilde{\mathbf g}^{\kappa}),$ where $\kappa = f(r)$ and the $G$-invariant metric $\tilde{\mathbf g}^{\kappa}$ is determined by 
\begin{equation}\label{fmain}
\tilde{\mathbf g}^{\kappa}_{o_{H}} (\cdot,\cdot) = \kappa^{2}\langle\cdot,\cdot\rangle_{\mathfrak a} + \frac{\kappa}{2}\langle\cdot,\cdot\rangle_{{\mathfrak m}_{\varepsilon}\oplus{\mathfrak k}_{\varepsilon}} + \frac{\kappa}{4}\langle\cdot,\cdot\rangle_{{\mathfrak m}_{\varepsilon/2}\oplus{\mathfrak k}_{\varepsilon/2}}.
\end{equation}
Moreover, $(\kappa^{-1}\xi,\tilde{\mathbf g}^{\kappa})$ is the unique $G$-invariant $K$-contact structure on $T_{r}(G/K)$ whose cha\-rac\-teristic vector field is $\kappa^{-1}\xi.$
\end{theorem}
We must point out that, from Proposition \ref{TJc}, the $G$-invariant almost Hermitian structure in Theorem \ref{main} cannot be differentially extended to all $T(G/K)$ and, from Proposition \ref{pstandard}, the standard contact structure $(\frac{1}{2}\eta,\frac{1}{4}\tilde{g}^{S})$ on $T_{1}{\mathbb S}^{n}$ or on $T_{1}{\mathbb R}{\mathbf P}^{n}$ is the Sasakian structure for $\kappa = \frac{1}{2}.$
\section{Preliminaries}
 \subsection{Almost contact metric structures}
We briefly recall some basic concepts of almost contact metric structures. For more information, see \cite{Bl}. An odd-dimensional smooth manifold $M$ is called {\em almost contact} if it admits a $(\varphi,\xi,\eta)$-structure, where $\varphi$ is a tensor field of type $(1,1),$ $\xi$ is a vector field and $\eta$ is a $1$-form, such that
\[
\varphi^{2} = -{\rm I} + \eta\otimes \xi,\quad \eta(\xi) = 1.
\]
Then $\varphi\xi = 0$ and $\eta\circ\varphi =0.$ If $M$ is equipped with a Riemannian metric $g$ such that
\[
g(\varphi X,\varphi Y) = g(X,Y) - \eta(X)\eta(Y),
\]
for all $X,Y\in {\mathfrak X}(M),$ where ${\mathfrak X}(M)$ is the Lie algebra of the vector fields on $M,$ $(M,\varphi,\xi,\eta,g)$ is said to be an {\em almost contact metric manifold} and $g$ is called a {\em compatible metric}. We say it is {\em contact metric} if $d\eta(X,Y) =g(X,\varphi Y).$ Then a contact metric structure is determined by the pair $(\xi,g)$ (or $(\eta,g))$ and $\xi$ is known as its {\em characteristic vector field}. If, in addition, $\xi$ is a Killing vector field, then the manifold is called {\em $K$-contact}. 

An almost contact structure $(\varphi,\xi,\eta)$ is said to be {\em normal} if the $(1,1)$-tensor field ${\mathbf N}(X,Y) = [\varphi,\varphi](X,Y) + 2{\rm d}\eta(X,Y)\xi$ vanishes for all $X,Y\in {\mathfrak X}(M),$ where $[\varphi,\varphi]$ is the {\em Nijenhuis torsion} of $\varphi,$ 
 \begin{equation}\label{Nijenhuis}
 [\varphi,\varphi] (X,Y) = \varphi^{2}[X,Y] + [\varphi X,\varphi Y] - \varphi[\varphi X,Y] - \varphi[X,\varphi Y].
 \end{equation}
 A contact metric structure $(\xi, g)$ which is normal is called a {\em Sasakian structure}. A useful characterization for Sasakian manifolds is the following: An almost contact metric manifold $(\varphi,\xi,\eta,g)$ is Sasakian if and only if  \begin{equation}\label{Sasakian}
 (\nabla_{X}\varphi)Y = g(X,Y)\xi - \eta(Y)X,\quad X,Y\in {\mathfrak X}(M),
 \end{equation}
 where $\nabla$ is the {\em Levi-Civita connection} of $(M,g).$ Any Sasakian manifold is always $K$-contact. The converse holds for the three-dimensional case, but it may not true in higher dimension.  
\subsection{Homogeneous manifolds}

A connected homogeneous manifold $M$ can be described as a quotient manifold $G/K,$ where $G$ is a Lie group, which is supposed to be connected, acting transitively and effectively on $M$ and $K$ is the isotropy subgroup of $G$ at some point $o\in M,$ the {\em origin} of $G/K.$ Denote by $\pi_{K}$ the projection $\pi_{K}\colon G\to G/K,$ $\pi_{K}(a) = aK,$ and by $\tau_{b},$ for each $b\in G,$ the translation $\tau_{b}\colon G/K\to G/K,$ $\tau_{b}(aK) = baK.$ If moreover $g$ is a $G$-invariant Riemannian metric on $M = G/K,$ then $(M,g)$ is said to be a {\em homogeneous Riemannian manifold}. 

When $G$ is compact, there exists a positive-definite ${\rm Ad}(G)$-invariant $2$-form $\langle\cdot,\cdot\rangle$ on the Lie algebra ${\mathfrak g}$ of $G$ and we have the reductive decomposition ${\mathfrak g} = {\mathfrak m}\oplus {\mathfrak k},$ ${\mathfrak m}$ being the $\langle\cdot,\cdot\rangle$-orthogonal complement of the Lie algebra ${\mathfrak k}$ of $K.$ Then the restriction of $\langle\cdot,\cdot\rangle$ to ${\mathfrak m}$ determines a $G$-invariant metric $g$ on $M.$ 

The differential map $(\pi_{K})_{*e}$ of $\pi_{K}$ at the identity element $e$ of $G$ gives a linear isomorphism of ${\mathfrak m}$ onto $T_{o}(G/K).$ Here and in what follows, $T_{o}(G/K)$ is identified with ${\mathfrak m}$ via $(\pi_{K})_{*e}.$ It is clear that there exists a sufficiently small neighborhood $O_{\mathfrak m}\subset {\mathfrak m}$ of zero in ${\mathfrak m}$ such that $\exp(O_{\mathfrak m})$ is a submanifold of $G$ and the mapping $\pi_{K\mid\exp(O_{\mathfrak m})}$ is a diffeomorphism onto a neighborhood ${\mathcal U}_{o}$ of $o.$ Then, for each $\xi\in {\mathfrak m},$ define $\xi^{\tau}$ as the vector field on ${\mathcal U}_{o}$ such that $\xi^{\tau}_{\exp x K} = (\tau_{\exp x})_{*o}\xi,$ for all $x\in O_{\mathfrak m}.$ Then $\xi^{\tau} = (\pi_{K})_{*}\xi^{\tt l},$ where $\xi^{\tt l}$ denotes the left $G$-invariant vector field on $G$ such that $\xi^{\tt l}_{e} = \xi.$ Under the identification $T_{o}(G/K)\cong {\mathfrak m},$ $\xi^{\tau}_{o}=\xi$ and $[\xi_{1}^{\tau},\xi^{\tau}_{2}]_{o} = [\xi_{1},\xi_{2}]_{\mathfrak m},$ where $[\cdot,\cdot]_{\mathfrak m}$ denotes the ${\mathfrak m}$-component of $[\cdot,\cdot]$ (see \cite{N}).

Let $\alpha\colon {\mathfrak m}\times {\mathfrak m}\to {\mathfrak m}$ be the ${\rm Ad}(K)$-invariant bilinear function that determines $\nabla,$ given by $\alpha(\xi_{1},\xi_{2}) = \nabla_{\xi_{1}}\xi_{2}^{\tau},$ for all $\xi_{1},\xi_{2}\in {\mathfrak m}.$ Then, using the Koszul formula, we have 
\begin{equation}\label{nabla}
\alpha(\xi_{1},\xi_{2}) = \frac{1}{2}[\xi_{1},\xi_{2}]_{\mathfrak m} + {\mathfrak U}(\xi_{1},\xi_{2}),
\end{equation} 
 where ${\mathfrak U}$ is the symmetric bilinear function on ${\mathfrak m}\times{\mathfrak m}$ such that
\begin{equation}\label{U}
2\langle{\mathfrak U}(\xi_{1},\xi_{2}),\xi_{3}\rangle = \langle[\xi_{3},\xi_{1}]_{\mathfrak m},\xi_{2}\rangle + \langle[\xi_{3},\xi_{2}]_{\mathfrak m},\xi_{1}\rangle.
\end{equation}
 When ${\mathfrak U} = 0,$ $(G/K,g)$ is said to be {\em naturally reductive}.  A $G$-invariant vector field $X$ on $G/K$ is Killing if and only if $\langle\alpha(\xi_{1},\xi),\xi_{2}\rangle + \langle\alpha({\xi_{2}},\xi),\xi_{1}\rangle = 0,$ for all $\xi_{1},\xi_{2}\in {\mathfrak m},$ where $\xi = X_{o}\in {\mathfrak m}.$ But, from (\ref{nabla}) and (\ref{U}), we get $\langle\alpha(\xi_{1},\xi),\xi_{2}\rangle + \langle\alpha({\xi_{2}},\xi),\xi_{1}\rangle = -2\langle{\mathfrak U}(\xi_{1},\xi_{2}),\xi\rangle.$ Hence, $X$ is a Killing vector field if and only if $\langle{\mathfrak U}(\cdot,\cdot),\xi\rangle = 0.$ In particular, on naturally reductive homogeneous manifolds any $G$-invariant vector field must be Killing.

\subsection{Tangent bundles}
Let $M = G/K$ be a homogeneous manifold. On the trivial vector bundle $G\times {\mathfrak m}$ consider two Lie group actions which commute on it: the left
$G$-action, $l_b \colon (a,x)\mapsto (ba,x)$ and the right
$K$-action $r_k \colon (a,x)\mapsto (ak,{\rm Ad}_{k^{-1}}x)$. 

Let
$\pi \colon G\times {\mathfrak m}\to G\times_K {\mathfrak m},$ $(a,x)\mapsto [(a,x)],$ be the natural projection for this right
$K$-action. Then $\pi$ is $G$-equivariant and, because the linear isotropy group $\{(\tau_{K})_{*o} \, : k\in K\}$ acting on $T_{o}(G/K)$ corresponds under $(\pi_{K})_{*e}$ with ${\rm Ad}(K)$ on ${\mathfrak m},$ the mapping $\phi$ given by
\begin{equation}\label{eq.phi}
\phi \colon G\times_K {\mathfrak m}\to T(G/K),
\quad
[(a,x)]\mapsto (\tau_{a})_{*o}x,
\end{equation}
is a $G$-equivariant diffeomorphism. 

Next, we determine the tangent space $T_{[(a,x)]}G\times_{K}{\mathfrak m}$ for all $(a,x)\in G\times {\mathfrak m}.$ For each $\xi\in {\mathfrak g},$ $\xi = \xi_{\mathfrak m} + \xi_{\mathfrak k}\in {\mathfrak m}\oplus{\mathfrak k},$ 
\[
\pi_{*(a,x)}((\xi_{\mathfrak k})^{\tt l}_{a},0) = \frac{d}{dt}_{\mid t = 0}\pi(a\exp t\xi_{\mathfrak k},x) = \frac{d}{dt}_{\mid t = 0}\pi(a,{\rm Ad}_{\exp t\xi_{\mathfrak k}}x) =\pi_{*(a,x)}(0,[\xi_{\mathfrak k},x]_{x}).
\]
Then, since $\pi_{*(a,x)}(\xi^{\tt l}_{g},u_{x})  =  \pi_{*(a,x)}(((\xi_{\mathfrak m})^{\tt l}_{a},u_{x}) +( (\xi_{\mathfrak k})^{\tt l}_{a},0)),$ it follows that 
\begin{equation}\label{kk}
\pi_{*(a,x)}(\xi^{\tt l}_{a},u_{x})  = \pi_{*(a,x)}((\xi_{\mathfrak m})^{\tt l}_{a},u_{x} + [\xi_{\mathfrak k},x]_{x})
\end{equation}
and we have
\begin{equation}\label{lGm}
 T_{[(a,x)]}G\times_{K}{\mathfrak m} = \{\pi_{*(a,x)}(\xi^{\tt l}_{a},u_{x})\colon (\xi,u)\in {\mathfrak m}\times {\mathfrak m}\},\quad (a,x)\in G\times{\mathfrak m}.
 \end{equation}
 
  \section{Compact rank-one symmetric spaces}
   
\subsection{Restricted roots of symmetric spaces of compact type}
We review a few facts about complex simple Lie algebras and symmetric spaces. See \cite[Ch. III and Ch. VII]{He} for more details. Suppose that ${\mathfrak g}$ is a compact simple Lie algebra and denote by ${\mathfrak g}^{\mathbb C}$ its complexification. If ${\mathfrak t}^{\mathbb C}$ is a Cartan subalgebra of ${\mathfrak g}^{\mathbb C},$ we have the root space decomposition
\[
{\mathfrak g}^{\mathbb C} = {\mathfrak t}^{\mathbb C} \oplus \sum_{\alpha\in \Delta}{\mathfrak g}_{\alpha},
\]
where $\Delta$ is the root system of ${\mathfrak g}^{\mathbb C}$ with respect to ${\mathfrak t}^{\mathbb C}$  and ${\mathfrak g}_{\alpha} = \{\xi\in {\mathfrak g}^{\mathbb C}\;:\;{\rm ad}_{t}\xi = \alpha(t)\xi,\;t\in {\mathfrak t}^{\mathbb C}\}.$ 

 Let $\Pi = \{\alpha_{1},\dots ,\alpha_{l}\}$ be a basis of $\Delta.$ Because the restriction of the Killing
form ${\mathbf B}$ of ${\mathfrak g}^{\mathbb C}$ to ${\mathfrak t}^{\mathbb
C} \times {\mathfrak t}^{\mathbb C}$ is nondegenerate, there
exists a unique element $t_{\alpha}\in {\mathfrak t}^{\mathbb C}$
such that ${\mathbf B}(t,t_{\alpha}) = \alpha(t),$ for all $t\in {\mathfrak t}^{\mathbb C}$ and ${\mathfrak t}^{\mathbb C}$ is expressed as ${\mathfrak t}^{\mathbb C} = \sum_{\alpha\in \Delta}{\mathbb
C}t_{\alpha}.$ Put $<\alpha
,\beta> ={\mathbf B}(t_{\alpha},t_{\beta}).$ 

Choose root vectors
$\{E_{\alpha}\}_{\alpha \in \Delta}$ such that $E_{\alpha}$ and $E_{\beta}$ are orthogonal under
${\mathbf B},$ if $\alpha + \beta \neq 0,$ and ${\mathbf B}(E_{\alpha},E_{-\alpha}) = 1$ and define the number $N_{\alpha,\beta}$ by $[E_{\alpha},E_{\beta}] = N_{\alpha,\beta}E_{\alpha + \beta},$ if $\alpha + \beta\in \Delta$ and $N_{\alpha,\beta} = 0$ if $\alpha + \beta \not\in
\Delta.$ Then
\begin{equation}\label{*}
N_{\alpha,\beta} = -N_{-\alpha,-\beta},\;\;\; N_{\alpha,\beta} =
-N_{\beta,\alpha}
\end{equation}
and, if $\alpha,\beta,\gamma\in \Delta$ and $\alpha + \beta +
\gamma = 0,$ 
\begin{equation}\label{**}
N_{\alpha,\beta} = N_{\beta,\gamma} = N_{\gamma,\alpha}.
\end{equation}
Moreover, given an $\alpha$-series $\beta + n\alpha$ $(p\leq n\leq
q)$ containing $\beta,$ 
\begin{equation}\label{***}
(N_{\alpha,\beta})^{2} = \frac{\textstyle q(1-p)}{\textstyle
2}<\alpha,\alpha>.
\end{equation}
Denote by $\Delta^{+}$ the set of positive roots of $\Delta$ with
respect to some lexicographic order in $\Pi.$ Then each $\alpha\in \Delta^{+}$ may be written as $\alpha = \sum_{k=1}^{l}n_{k}(\alpha)\alpha_{k},$ where $n_{k}(\alpha)\in {\mathbb Z},$ $n_{k}(\alpha)\geq 0,$ for all $k = 1,\dots, l,$ and ${\mathfrak g}$ as subspace of ${\mathfrak g}^{\mathbb
C}$ is given by
\[
{\mathfrak g} = {\mathfrak t} \oplus \displaystyle\sum_{\alpha \in
\Delta^{+}} ({\mathbb R}\; U^{0}_{\alpha} + {\mathbb R}\; U^{1}_{\alpha}),
\]
 where ${\mathfrak t} = \sum_{\alpha\in \Delta}{\mathbb R} {\rm i} t_{\alpha},$ $U^{0}_{\alpha} = E_{\alpha}-E_{-\alpha}$ and $U^{1}_{\alpha}
= {\mathrm i}(E_{\alpha} + E_{-\alpha}).$ Using \cite[Theorem 5.5, Ch. III]{He}, we obtain the following.
\begin{lemma}\label{bracket} For all $\alpha ,\beta \in
\Delta^{+}$ and $a = 0,1,$ the following equalities hold:
\begin{enumerate}
\item[{\rm (i)}]
$[U^{a}_{\alpha},i t_{\beta}] = (-1)^{a+1}<\alpha,\beta>
U^{a+1}_{\alpha};$
\item[{\rm (ii)}] $[U^{0}_{\alpha},U^{1}_{\alpha}] =
2i t_{\alpha};$
\item[{\rm (iii)}]
$[U^{a}_{\alpha},U^{b}_{\beta}] =
(-1)^{ab}N_{\alpha,\beta}U^{a+b}_{\alpha + \beta} +
(-1)^{a+b}N_{-\alpha,\beta}U^{a+b}_{\alpha -\beta},$ where $\alpha
\neq \beta$ and $a\leq b.$
\end{enumerate}
\end{lemma}

Let $\mu = \sum_{i=1}^{l}m_{i}\alpha_{i}$ be the {\em maximal
root} of $\Delta$ and consider $t_{j}\in {\mathfrak t}^{\mathbb
C},$ $j=1,\dots ,l,$ defined by $\alpha_{k}(t_{j}) = (1/m_{j})\delta_{jk},$ $ j,k = 1,\dots ,l.$ Each inner automorphism of order $2$ on ${\mathfrak g}^{\mathbb
C}$ is conjugate in the inner automorphism group of ${\mathfrak
g}^{\mathbb C}$ to some $\sigma = Ad_{\exp 2\pi {\rm i} t},$ where
$t = (1/2)t_{k}$ with $m_{k}=1$ (Hermitian symmetric space) or $t = t_{k}$ with $m_{k} = 2$ (non-Hermitian symmetric space). Denote by $\Delta^{+}(t)$ the positive root system
generated by $\Pi(t),$ where $\Pi(t) = \{\alpha_{j}\in \Pi\colon j\neq k\},$ if $m_{k} = 1,$ and $\Pi(t) = \{\alpha_{j}\in \Pi\colon j\neq k\}\cup\{-\mu\},$ if $m_{k} = 2.$ Then, ${\mathfrak t}\subset
{\mathfrak k},$ or equivalently ${\rm rank}\;{\mathfrak g} = {\rm rank}\;{\mathfrak k},$ and 
\begin{equation}\label{kk1}
{\mathfrak k} = {\mathfrak t} \oplus \displaystyle\sum_{\alpha \in
\Delta^{+}(t)}({\mathbb R}\; U^{0} _{\alpha}+ {\mathbb R}\; U^{1}_{\alpha}).
\end{equation}
Because ${\mathbf B}(U^{a}_{\alpha},U^{b}_{\beta}) =
-2\delta_{\alpha\beta}\delta_{ab},$ it follows that the set
$\{U^{a}_{\alpha}\colon a=0,1;\, \alpha\in
\Delta^{+}\setminus\Delta^{+}(t)\}$ is an
orthonormal basis for $({\mathfrak m}, -\frac{1}{2}{\mathbf B}_{\mid
{\mathfrak m}}).$

If $G/K$ is a symmetric spaces of compact type, then $G$ is a compact semisimple Lie group and there exists an involutive automorphism $\sigma\colon {\mathfrak g}\to {\mathfrak g}$ such that the corresponding $\pm 1$-eigenspaces ${\mathfrak k} =\{\xi\in {\mathfrak g}\colon \sigma(\xi) = \xi\}$ and ${\mathfrak m} = \{\xi\in {\mathfrak g}\colon \sigma(\xi) = -\xi\}$ determine a reductive decomposition ${\mathfrak g} = {\mathfrak m}\oplus {\mathfrak k}$ with $[{\mathfrak m},{\mathfrak m}]\subset {\mathfrak k}.$

 Suppose that $G/K$ has rank ${\mathbf r}.$ Then, fixed a ${\mathbf r}$-dimensional Cartan subspace ${\mathfrak a}$ of ${\mathfrak m},$ there exists a $\sigma$-invariant Cartan subalgebra ${\mathfrak t}_{\mathfrak a}$ of ${\mathfrak g}$ containing ${\mathfrak a}.$ This implies that  ${\mathfrak t}_{\mathfrak a} = {\mathfrak a}\oplus {\mathfrak t}_{0},$ where ${\mathfrak t}_{0} = {\mathfrak t}_{\mathfrak a}\cap {\mathfrak k},$ and the complexification ${\mathfrak t}_{\mathfrak a}^{\mathbb C}$ is a Cartan subalgebra of ${\mathfrak g}^{\mathbb C}.$ 

To the set $\Sigma = \{\lambda\in ({\mathfrak a}^{\mathbb C})^{*}\colon\lambda = \alpha_{\mid {\mathfrak a}^{\mathbb C}},\,\alpha\in \Delta\setminus \Delta_{0}\}$ of non-vanishing restrictions of roots to ${\mathfrak a}^{\mathbb C},$ where $\Delta_{0} = \{\alpha\in \Delta\colon\alpha_{\mid {\mathfrak a}^{\mathbb C}} = 0\},$ is known as the set of {\em restricted roots} of $({\mathfrak g},{\mathfrak k},{\mathfrak a}).$ Denote by $\Sigma^{+}$ the subset of positive restricted roots in $\Sigma$ determined $\Delta^{+}$ and by $m_{\lambda}$ the multiplicity of $\lambda\in \Sigma^{+},$ that is, $m_{\lambda} = {\rm card}\{\alpha\in \Delta\colon \alpha_{\mid {\mathfrak a}^{\mathbb C}} = \lambda\}.$ For each linear form $\lambda$ on ${\mathfrak a}^{\mathbb C}$ put
\[
{\mathfrak m}_\lambda =
\big\{\eta\in{\mathfrak m}:\operatorname{ad}^2_w(\eta)
=\lambda^2(w)\eta,\ \forall w\in{\mathfrak a}\big\}, \quad
{\mathfrak k}_\lambda =
\big\{\zeta\in{\mathfrak k}: \operatorname{ad}^2_w(\zeta)=\lambda^2(w)\zeta,\
\forall w\in{\mathfrak a}\big\}.
\]
Then ${\mathfrak m}_{\lambda}={\mathfrak m}_{-\lambda}$,
${\mathfrak k}_{\lambda}={\mathfrak k}_{-\lambda}$,
${\mathfrak m}_0={\mathfrak a}$ and ${\mathfrak k}_0$ equals to the centralizer ${\mathfrak h}$ of ${\mathfrak a}$ in
${\mathfrak k},$ that is,
\[
{\mathfrak h} = \{u\in {\mathfrak k}\;\colon [u,{\mathfrak a}] = 0\}.
\]
Clearly ${\mathfrak t}_{0}$ is a maximal abelian subalgebra of ${\mathfrak h}$ and ${\rm rank}\; {\mathfrak h} = {\rm rank}\;{\mathfrak g} - {\mathbf r}.$ Put ${\mathfrak h} = {\mathfrak h}_{0}\oplus {\mathfrak h}_{1},$ where ${\mathfrak h}_{0} = {\mathfrak t}_{0}$ and ${\mathfrak h}_{1}$ is the orthogonal complement of ${\mathfrak h}_{0}$ in ${\mathfrak h}.$ By~\cite[Ch.\ VII, Lemma 11.3]{He}, the following
decompositions are direct and orthogonal:
\begin{equation}\label{eq.ms4.2}
{\mathfrak m}={\mathfrak a}\oplus\sum_{\lambda\in\Sigma^+}{\mathfrak m}_\lambda,
\qquad
{\mathfrak k}={\mathfrak h}\oplus
\sum_{\lambda\in\Sigma^+}{\mathfrak k}_\lambda.
\end{equation}

Define the linear function $\lambda_{\mathbb R} \colon {\mathfrak a}\to{\mathbb R}$,
$\lambda\in\Sigma^+$, by the relation ${\rm i}\lambda_{\mathbb R}=\lambda$. Note that, since the Lie algebra ${\mathfrak g}$ is compact, $\lambda({\mathfrak a})\subset {\rm i} {\mathbb R}$. From \cite[Lemma 4.2]{GGM1} we have that for any vector
$\xi_\lambda\in{\mathfrak m}_\lambda$,
$\lambda\in\Sigma^+$, there exists a unique vector
$\zeta_\lambda\in{\mathfrak k}_\lambda$ such that 
\begin{equation}\label{eq.ms4.3}
[u,\xi_\lambda]=  -\lambda_{\mathbb R}(u)\zeta_\lambda,
\quad [u,\zeta_\lambda]= \lambda_{\mathbb R}(u)\xi_\lambda,
\qquad\text{for all}\ u\in{\mathfrak a}.
\end{equation}

 \subsection{Restricted roots of compact rank-one symmetric spaces}\label{subrestricted}  Compact rank-one Riemannian symmetric spaces can be characterized as those symmetric spaces with strictly positive curvature. They are Euclidean spheres, real, complex and quaternionic projective spaces and the Cayley plane \cite[Ch. 3]{Be1}. Their quotient expressions $G/K$ as symmetric spaces appear in Table I.
   
 \bigskip
\begin{tabular}{llllll}
\multicolumn{6}{c}{Table I. Compact rank-one symmetric spaces}\\
\hline\noalign{\smallskip}
& $G/K$
& {$\dim$}
& $m_\varepsilon$
& $m_{\varepsilon/{\scriptscriptstyle 2}}$ & ${\mathfrak h}$ \\
\noalign{\smallskip}\hline\noalign{\smallskip}
$\begin{matrix}
\mathbb{S}^n,(n\geq 2) \\ \noalign{\smallskip}
 {\mathbb R}{\mathbf P}^n, (n\geq 2)
\end{matrix}$ &
$\begin{matrix}
\hspace{-1.5cm}\mathrm{SO}(n{+}1)/\mathrm{SO}(n) \\ \noalign{\smallskip}
 \mathrm{SO}(n{+}1)/{\rm S}({\rm O}(1)\times {\rm O}(n))\end{matrix}$ \smallskip&
$n$ & $n{-}1 $ & $0$ & $\mathfrak{so}(n{-}1)$ \\ \smallskip
${\mathbb C}{\mathbf P}^n$,\,{$(n \geq 2)$}
&{$\mathrm{SU}(n{+}1)/\mathrm{S}(\mathrm{U}(1){\times}
\mathrm{U}(n))$}& $2n$
& $1$ & $2n{-}2$ & ${\mathbb R}\oplus \mathfrak{su}(n{-}1)$ \\ \smallskip
${\mathbb H}{\mathbf P}^n$,\,{$(n\geq 1)$}
& {$\mathrm{Sp}(n{+}1)/\mathrm{Sp}(1){\times} \mathrm{Sp}(n)$} & $4n$
& $3$ & $4n{-}4$ & $\mathfrak{sp}(1)\oplus \mathfrak{sp}(n{-}1)$ \\
\smallskip
${\mathbb C\mathrm a}{\mathbf P}^2$
& {$\mathrm{F_4}/\mathrm{Spin(9)}$} & $16$ & $7$
& 8 & $\mathfrak{so}(7)$ \\
\hline
\end{tabular}

\vspace{0.5cm}

Except for odd-dimensional spheres ${\mathbb S}^{2n-1} = G/K = \SO(2n)/\SO(2n-1)$ $(n\geq 2),$ where ${\rm rank}\,{\mathfrak g} = n>{\rm rank}\,{\mathfrak k} = n-1,$ the rest of all compact rank-one symmetric spaces admit inner automorphisms. Note that $\dim{\mathfrak a} = 1$ and, from Lemma \ref{bracket} and (\ref{kk1}), a compact symmetric space $G/K$ has rank one if and only if $\alpha + \beta \neq 0$ or $\alpha -\beta\neq 0,$ for all $\alpha,\beta\in \Delta^{+}\setminus \Delta^{+}(t),$ $\alpha\neq \beta.$ 

Each compact rank-one symmetric is described in detail below. As a consequence, the following lemmas, Lemma \ref{restricted} and Lemma \ref{pbrack}, are proven.
  \begin{lemma}\label{restricted} The restricted root system $\Sigma$ of $({\mathfrak g},{\mathfrak k},{\mathfrak a})$  is either $\Sigma = \{\pm\varepsilon\}$ or $\Sigma = \{\pm\varepsilon,\pm\frac{1}{2}\varepsilon\},$ where $\varepsilon\in ({\mathfrak a}^{\mathbb C})^{*}.$ The corresponding multiplicities $m_\varepsilon$, $m_{\varepsilon/{\scriptscriptstyle 2}}$ of these roots and the Lie algebra ${\mathfrak h}$ are listed in {\rm Table I}.
 \end{lemma}
   Denote by $X$ the unique (basis) vector $X\in {\mathfrak a}$ such that $\varepsilon_{\mathbb R}(X) = 1,$ where $\varepsilon = {\rm i}\varepsilon_{\mathbb R}.$ Multiplying the inner product $\langle\cdot,\cdot\rangle$ by a positive constant, we can assume that $\langle X,X\rangle = 1.$ 
  
  In order to simplify notations we put ${\mathfrak m}^{+} = {\mathfrak m}_{\varepsilon}\oplus {\mathfrak m}_{\varepsilon/2}$ and ${\mathfrak k}^{+} = {\mathfrak k}_{\varepsilon}\oplus {\mathfrak k}_{\varepsilon/2},$ where ${\mathfrak m}_{\varepsilon/2}= 0$ and ${\mathfrak k}_{\varepsilon/2} =0$ if $\varepsilon/2$ is not in $\Sigma.$ Then fixing in ${\mathfrak m}_{\varepsilon}$ and ${\mathfrak m}_{\varepsilon/2}$ some $\langle\cdot,\cdot\rangle$-orthonormal basis $\{\xi^{j}_{\varepsilon}\}$ and $\{\xi^{p}_{\varepsilon/2}\},$ $j = 1,\dots, m_{\varepsilon},$ $p = 1,\dots, m_{\varepsilon/2},$ we take the unique bases $\{\zeta^{j}_{\varepsilon}\}$ and $\{\zeta^{p}_{\varepsilon/2}\}$ of ${\mathfrak k}_{\varepsilon}$ and ${\mathfrak k}_{\varepsilon/2}$ satisfying (\ref{eq.ms4.3}) for $\lambda = \varepsilon$ and $\lambda = \varepsilon/2.$ Clearly these last two basis are also $\langle\cdot,\cdot\rangle$-orthonormal.  
 \begin{lemma}\label{pbrack} We have
\begin{equation}\label{S2}
[\xi^{j}_{\varepsilon},\xi^{p}_{\varepsilon/2}] = [\zeta^{j}_{\varepsilon},\zeta^{p}_{\varepsilon/2}],\quad [\zeta^{j},\xi^{p}_{\varepsilon/2}] = -[\xi^{j}_{\varepsilon},\zeta^{p}_{\varepsilon/2}],
\end{equation}
for all $j = 1,\dots, m_{\varepsilon}$ and $p = 1,\dots, m_{\varepsilon/2}.$
\end{lemma}  
  
Applying \cite[Ch. VII, Lemma 11.4, p. 335]{He}), it is immediate the following.
\begin{lemma}\label{lbrack1} We have:
\begin{equation}\label{brack1}
\begin{array}{l}
[{\mathfrak h},{\mathfrak m}_{\lambda}]\subset {\mathfrak m}_{\lambda},\quad [{\mathfrak h},{\mathfrak k}_{\lambda}]\subset{\mathfrak k}_{\lambda},\quad [{\mathfrak a},{\mathfrak m}_{\lambda}]\subset{\mathfrak k}_{\lambda},\quad [{\mathfrak a},{\mathfrak k}_{\lambda}]\subset {\mathfrak m}_{\lambda},\;\; \lambda = \varepsilon,\varepsilon/2,\\[0.4pc]
[{\mathfrak m}_{\varepsilon},{\mathfrak m}_{\varepsilon}]\subset {\mathfrak h},\quad [{\mathfrak m}_{\varepsilon},{\mathfrak m}_{\varepsilon/2}]\subset {\mathfrak k}_{\varepsilon/2},\quad [{\mathfrak m}_{\varepsilon},{\mathfrak k}_{\varepsilon}]\subset {\mathfrak a},\quad [{\mathfrak m}_{\varepsilon},{\mathfrak k}_{\varepsilon/2}]\subset {\mathfrak m}_{\varepsilon/2},\\[0.4pc]
[{\mathfrak m}_{\varepsilon/2},{\mathfrak m}_{\varepsilon/2}]\subset {\mathfrak h}\oplus {\mathfrak k}_{\varepsilon},\quad [{\mathfrak m}_{\varepsilon/2},{\mathfrak k}_{\varepsilon}]\subset{\mathfrak m}_{\varepsilon/2},\quad [{\mathfrak m}_{\varepsilon/2},{\mathfrak k}_{\varepsilon/2}]\subset {\mathfrak a}\oplus {\mathfrak m}_{\varepsilon},\\[0.4pc]
[{\mathfrak k}_{\varepsilon},{\mathfrak k}_{\varepsilon}]\subset {\mathfrak h},\quad [{\mathfrak k}_{\varepsilon},{\mathfrak k}_{\varepsilon/2}]\subset {\mathfrak k}_{\varepsilon/2},\quad [{\mathfrak k}_{\varepsilon/2},{\mathfrak k}_{\varepsilon/2}]\subset {\mathfrak h}\oplus {\mathfrak k}_{\varepsilon}.
\end{array}
\end{equation}
\end{lemma}\vspace{0.2cm}
\noindent {\bf Case 1: ${\mathbb S}^{n},$ ${\mathbb R}{\mathbf P}^{n}$} $(n\geq 2).$ Denote by $E_{jk}$ the $(n+1)\times (n+1)$-matrix with $1$ in the entry of the $j$th rwo and the $k$th column and $0$ elsewhere. Then the set of matrices $\{A_{jk} = E_{jk} - E_{kj}\colon 1\leq j<k\leq n+1\}$ is an orthonormal basis of the Lie algebra ${\mathfrak s}{\mathfrak o}(n+1)$ of $\SO(n+1)$ with respect to the invariant trace form $\langle A,B\rangle = -\frac{1}{2}{\rm trace AB},$ for all $A,B\in {\mathfrak s}{\mathfrak o}(n+1).$ Here, ${\mathfrak k} = {\mathbb R}\{A_{jk}\}_{2\leq j<k\leq n+1}$ is isomorphic to ${\mathfrak s}{\mathfrak o}(n)$ and ${\mathfrak m}= {\mathbb R}\{A_{1k}\}_{2\leq k\leq n+1}$ is its orthogonal complement in ${\mathfrak s}{\mathfrak o}(n+1).$ We take as Cartan subspace to ${\mathfrak a} = {\mathbb R}X,$ where $X= A_{12}.$ Then the centralizer ${\mathfrak h}$ of ${\mathfrak a}$ in ${\mathfrak k}$ is given by
\begin{equation}\label{hh}
{\mathfrak h} = {\mathbb R}\{A_{jk}\}_{3\leq j<k\leq n+1}.
\end{equation}
 Because ${\rm ad}^{2}_{X}A_{1k} = -A_{1k},$ for $k = 3,\dots ,n+1,$ it follows that ${\mathfrak m}_{\varepsilon} ={\mathbb R}\{A_{1k}\}_{k = 3,\dots, n+1}.$ This implies that $\Sigma^{+} = \{\varepsilon\},$  $m_{\varepsilon} = n-1$ and $m_{\varepsilon/2} = 0.$ 

\vspace{0.1cm}

\noindent{\bf Case 2: ${\mathbb C}{\mathbf P}^{n}$} $(n\geq 2).$ Consider ${\mathfrak g} = {\mathfrak s}{\mathfrak u}(n+1)$ as compact real form of the complex simple Lie algebra ${\mathfrak a}_{n} = {\mathfrak s}{\mathfrak l}(n+1,{\mathbb C}).$ Let $\Pi = \{\alpha_{1},\dots ,\alpha_{n}\}$ be a system of simple roots. Then the set $\Delta^{+}$ of positive roots is expressed as 
 \[
 \Delta^{+} = \{\alpha_{pq} = \alpha_{p}+\dots + \alpha_{q}\;\colon 1\leq p \leq q \leq n\}.
 \]
 We get $\langle\alpha_{1},\alpha_{1}\rangle = \dots = \langle\alpha_{n},\alpha_{n}\rangle = -\frac{1}{2}\langle\alpha_{k},\alpha_{k+1}\rangle,$ for all $k = 1,\dots, n-1,$ being zero the rest of the products of basic elements.
 
  The complex projective space ${\mathbb C}{\mathbf P}^{n}$ is obtained for the inner automorphism $\sigma = {\rm Ad}_{\exp 2\pi i t},$ where $t = \frac{1}{2}t_{1}$ (or, equivalently, $t = \frac{1}{2}t_{n}).$ Then $\Pi(t) = \{\alpha_{2},\dots ,\alpha_{n}\},$ $\Delta^{+}(t) = \{\alpha_{pq}\colon 2\leq p\leq q \leq n\}$ and $\Delta^{+}\setminus \Delta^{+}(t) = \{\alpha_{1q}\colon 1\leq q\leq n\}.$ We choose as Cartan subspace to ${\mathfrak a} = {\mathbb R} U^{0}_{\alpha_{1}}.$ 
 
 From Lemma \ref{bracket} it follows that $[U^{0}_{\alpha_{1}}, i t_{\alpha_{1}} + i t_{\alpha_{2}}] = 0$ and $[U^{0}_{\alpha_{1}},i t_{\alpha_{k}}] = 0,$ for all $k = 3,\dots ,n.$ Therefore, the Cartan subalgebra ${\mathfrak t}^{\mathbb C}_{\mathfrak a}$ of ${\mathfrak a}_{n}$ is the complexification of ${\mathfrak t}_{\mathfrak a} = {\mathfrak a}\oplus {\mathfrak t}_{0},$ where ${\mathfrak t}_{0} = {\mathbb R}\{i t_{\alpha_{1}} + 2 i t_{\alpha_{2}},\; i t_{\alpha_{k}}\;(k = 3,\dots ,n)\}.$
 
  Put $X= \beta U^{0}_{\alpha_{1}},$ where $\beta = (2\langle \alpha_{1},\alpha_{1}\rangle)^{-1}.$ Then, using Lemma \ref{bracket} and (\ref{*}), (\ref{**}) and (\ref{***}), we have
  \[
 {\rm ad}^{2}_{X}U^{1}_{\alpha_{1}} = -U^{1}_{\alpha_{1}},\quad {\rm ad}^{2}_{X}U^{a}_{\alpha_{1(j+1)}} = -\frac{1}{4}U^{a}_{\alpha_{1(j+1)}},
 \]
 for all $a = 0,1$ and $j =1,\dots, n-1.$ Hence, ${\mathfrak m}_{\varepsilon} = {\mathbb R}U^{1}_{\alpha_{1}}$ and ${\mathfrak m}_{\varepsilon/2}= {\mathbb R}\{U^{a}_{\alpha_{1(j+1)}}\colon a = 0,1;\;1\leq j\leq n-1\}.$ Since $\dim {\mathfrak m}_{\varepsilon/2} = 2(n-1),$ it follows that $\Sigma^{+} = \{\varepsilon, \frac{1}{2}\varepsilon\}.$ 
 
 We choose the inner product $\langle\cdot,\cdot\rangle$ such that on ${\mathfrak m}$ it is proportional to ${\mathbf B}_{\mid {\mathfrak m}}$ and $\langle X,X\rangle = 1.$ Then, $\langle \cdot,\cdot \rangle = -\langle\alpha_{1},\alpha_{1}\rangle {\mathbf B}(\cdot,\cdot).$ Hence $\xi_{\varepsilon} = \beta U^{1}_{\alpha_{1}}$ is a unit vector of ${\mathfrak m}_{\varepsilon}$ and the set of vectors
 \[
 \{\xi^{j,a}_{\varepsilon/2} = \beta U^{a}_{\alpha_{1(j+1)}}\colon\; a = 0,1,\; j = 1,\dots, n-1\}
 \]
 forms an orthonormal basis for ${\mathfrak m}_{\varepsilon/2}.$ Applying Lemma \ref{bracket} and  (\ref{eq.ms4.3}), ${\mathfrak k}_{\varepsilon} = {\mathbb R}\zeta_{\varepsilon}$ and ${\mathfrak k}_{\varepsilon/2} = {\mathbb R}\{\zeta^{j,a}_{\varepsilon/2}\colon \, a = 0,1,\; j = 1,\dots, n-1\},$ where $\zeta_{\varepsilon} = -{\rm i}\langle\alpha_{1},\alpha_{1}\rangle^{-1} t_{\alpha_{1}}$ and $\zeta^{j,a}_{\varepsilon/2} = \beta U^{a}_{\alpha_{2(j+1)}}.$
 
 Let $\Delta^{+}_{1}(t) = \{\alpha\in \Delta^{+}(t)\colon \alpha\pm \alpha_{1}\not\in\Delta\} = \{\alpha_{pq}\colon 3\leq p\leq q \leq n\}.$ Then ${\mathfrak h} = {\mathfrak h}_{0} \oplus {\mathfrak h}_{1},$ where ${\mathfrak h}_{0} = {\mathfrak t}_{0}$ and ${\mathfrak h}_{1} = \sum_{\alpha\in \Delta^{+}_{1}(t)}
 ({\mathbb R} U^{0}_{\alpha} + {\mathbb R}U^{1}_{\alpha}).$ From Lemma \ref{bracket}, the vector $Z_{1}$ given by
 \begin{equation}\label{Z1}
  Z_{1} = (n-1)(i t_{\alpha_{1}} + 2 i t_{\alpha_{2}}) + 2((n-2)i t_{\alpha_{3}} + (n-3)i t_{\alpha_{4}} +\dots + i t_{\alpha_{n}})
  \end{equation}
  belongs to the center ${\mathfrak z}({\mathfrak h})$ of ${\mathfrak h}$ and 
  \begin{equation}\label{4.27}
  {\mathfrak h} = {\mathbb R}Z_{1} \oplus {\mathfrak h}_{s} ,
   \end{equation}
   where ${\mathfrak h}_{s} = {\mathbb R}\{{\rm i}t_{\alpha_{k}}\colon k = 3,\dots, n\}\oplus {\mathfrak h}_{1} = {\mathfrak s}{\mathfrak u}(n-1)$ is the maximal semisimple ideal $[{\mathfrak h},{\mathfrak h}]$ of ${\mathfrak h}.$ By direct calculation, using (\ref{eq.ms4.3}) and (\ref{brack1}), the nonzero brackets of basic elements of ${\mathfrak m}^{+}\oplus {\mathfrak k}^{+}$ satisfy the following equalities:
 \begin{equation}\label{brackTCP}
 \begin{array}{lcl}
  [\xi_{\varepsilon},\xi^{j,a}_{\varepsilon/2}]  =  [\zeta_{\varepsilon},\zeta^{j,a}_{\varepsilon/2}] = \frac{(-1)^{a}}{2}\zeta^{j,a+1}_{\varepsilon/2}, & &[\xi_{\varepsilon}, \zeta^{j,a}_{\varepsilon/2}] = -[\zeta_{\varepsilon}, \xi^{j,a}_{\varepsilon/2}] = \frac{(-1)^{a}}{2}\xi^{j,a+1}_{\varepsilon/2},\\[0.5pc] 
  [\xi^{j}_{\varepsilon/2},\zeta^{j}_{\varepsilon/2}]  =  [\xi^{j,1}_{\varepsilon/2},\zeta^{j,1}_{\varepsilon/2}] = -\frac{1}{2}X, & & [\xi^{j}_{\varepsilon/2}, \zeta^{j,1}_{\varepsilon/2}] = -[\xi^{j,1}_{\varepsilon/2}, \zeta^{j}_{\varepsilon/2}] = \frac{1}{2}\xi_{\varepsilon},\\[0.5pc] 
 [\xi_{\varepsilon},\zeta_{\varepsilon}] = -X, & & [\xi^{j}_{\varepsilon/2},\xi^{j,1}_{\varepsilon/2}]_{\overline{\mathfrak m}} = -
[\zeta^{j}_{\varepsilon/2},\zeta^{j,1}_{\varepsilon/2}]_{\overline{\mathfrak m}} = -\frac{1}{2}\zeta_{\varepsilon}.
  \end{array}
 \end{equation}
 
 \vspace{0.1cm}
 
 \noindent {\bf Case 3: ${\mathbb H}{\mathbf P}^{n}$} $(n\geq 1).$ The Lie algebra ${\mathfrak g} = \asp(n+1)$ of $\Sp(n+1)$ is a compact real form of ${\mathfrak c}_{n+1} = \asp(n+1,{\mathbb C}),$
$\xymatrix@R=.5cm@C=.8cm{ \stackrel{2}{\stackrel{\circ}{\alpha_{1}}}
\ar@{-}[r] & \stackrel{2}{\stackrel{\circ}{\alpha_{2}}} \ar@{-}[r] &
\; \dots \ar@{-}[r] & \stackrel{2}{\stackrel{\circ}{\alpha_{n}}}
\ar@2{<-}[r] & \stackrel{1}{\stackrel{\circ}{\alpha_{n+1}}}}
$ and 
\[
\Delta^{+}  =  \{\alpha_{pq} \; (1\leq p\leq q\leq n+1),
\widetilde{{\alpha}_{pq}} \; (1\leq p\leq q\leq n)\}
\]
is a set of positive roots, where $\widetilde{{\alpha}_{pq}}= \alpha_{p} +\dots + 2\alpha_{q} + \dots +
2\alpha_{n} + \alpha_{n+1},$ for $1\leq p< q\leq n,$ and $\widetilde{\alpha_{pp}} = 2\alpha_{p} + \dots + 2\alpha_{n}+ \alpha_{n+1},$ for $1\leq p\leq n.$ The maximal root is $\mu = \widetilde{\alpha_{11}}$ and the inner product $\langle\cdot,\cdot\rangle$ on $\Pi = \{\alpha_{1},\dots,\alpha_{n+1}\}$ satisfies
\[
\langle\alpha_{1},\alpha_{1}\rangle = \dots = \langle\alpha_{n},\alpha_{n}\rangle  = \frac{1}{2}\langle\alpha_{n+1},\alpha_{n+1}\rangle,\quad \langle\alpha_{1},\alpha_{1}\rangle = -\langle\alpha_{n},\alpha_{n+1}\rangle = -\frac{1}{2}\langle\alpha_{j},\alpha_{j+1}\rangle,
\]
$1\leq j\leq n-1,$ being zero the rest of products of basic elements.

 The quaternionic projective space ${\mathbb H}{\mathbf P}^{n}$ appears for $\sigma = {\rm Ad}_{\exp 2\pi{\rm i}t},$ where $t = t_{1}.$ Then, $\Pi(t) = \{-\mu,\alpha_{2},\dots,\alpha_{n+1}\}.$ Hence, $\Delta^{+}(t) = \{\alpha_{p,q}\;(2\leq p\leq q \leq n+1),\; \widetilde{\alpha_{p,q}}\;(2\leq p\leq q \leq n),\mu\}$ and $\Delta^{+}\setminus \Delta^{+}(t) = \{\alpha_{1q}\;(1\leq q \leq n+1),\widetilde{\alpha_{1q}}\;(2\leq q\leq n)\}.$
 
 From (\ref{kk1}), ${\mathfrak k} = {\mathfrak s}{\mathfrak p}(1)\oplus {\mathfrak s}{\mathfrak p}(n),$ where ${\mathfrak s}{\mathfrak p}(1)\cong{\mathbb R}\{{\rm i}t_{\alpha_{\mu}}, U^{0}_{\mu},U^{1}_{\mu}\}.$ By Lemma \ref{bracket}, because $\mu\pm \alpha \notin\Delta$ and $\langle\mu,\alpha\rangle = 0,$ for all $\alpha \in \Delta^{+}(t),$ this subalgebra is the center ${\mathfrak z}({\mathfrak k})$ of ${\mathfrak k}.$

We choose as Cartan subalgebra to ${\mathfrak a} = {\mathbb R}U^{0}_{\alpha_{1}}.$ From Lemma \ref{bracket}, it follows that the Cartan subalgebra ${\mathfrak t}^{\mathbb C}_{\mathfrak a}$ of ${\mathfrak c}_{n+1}$ containing ${\mathfrak a}$ is the complexification of ${\mathfrak t}_{\mathfrak a} = {\mathfrak a} \oplus {\mathfrak t}_{0},$ where ${\mathfrak t}_{0} = {\mathbb R}\{{\rm i}t_{\alpha_{1}} + 2{\rm i}t_{\alpha_{2}},{\rm i}t_{\alpha_{k}}\; (k = 3,\dots, n+1)\}.$ Put $X = \beta U^{0}_{\alpha_{1}},$ where $\beta = (2\langle\alpha_{1},\alpha_{1}\rangle)^{-1/2}.$ Applying again Lemma \ref{bracket} and using (\ref{*}), (\ref{**}) and (\ref{***}), we have
$$
\begin{array}{lclclcl}
{\rm ad}^{2}_{X}U^{1}_{\alpha_{1}} & =  & -U^{1}_{\alpha_{1}},& &{\rm ad}^{2}_{X}U^{a}_{\widetilde{\alpha_{12}}} & = & -U^{a}_{\widetilde{\alpha_{12}}},\\[0.4pc]
{\rm ad}^{2}_{X}U^{a}_{\alpha_{1(j+1)}} & = & -\frac{1}{4}U^{a}_{\alpha_{1
(j+1)}},& & {\rm ad}^{2}_{X}U^{a}_{\widetilde{\alpha_{1(k+2)}}} & = & -\frac{1}{4}U^{a}_{\widetilde{\alpha_{1(k+2)}}},
\end{array}
$$
for $a = 0,1,$ $j = 1,\dots,n$ and $k = 1,\dots,n-2.$ We choose the inner product $\langle \cdot,\cdot \rangle = -\langle\alpha_{1},\alpha_{1}\rangle {\mathbf B}(\cdot,\cdot).$ Then $\{\xi_{\varepsilon}^{1},\xi^{2}_{\varepsilon},\xi^{3}_{\varepsilon}\}$ is an $\langle\cdot,\cdot\rangle$-orthonormal basis of ${\mathfrak m}_{\varepsilon},$ where $\xi^{1}_{\varepsilon} = \beta U^{1}_{\alpha_{1}}$ and $\xi^{2 +a}_{\varepsilon} = \beta U^{a}_{\widetilde{\alpha_{12}}},$ $a = 0,1,$ and the set of vectors
\[
 \xi^{j,a}_{\varepsilon/2}  = \beta U^{a}_{\alpha_{1(j+1)}}\;\;(j = 1,\dots, n), \quad \xi^{n+k,a}_{\varepsilon/2} =  \beta U^{a}_{\widetilde{\alpha_{1(k+2)}}}\; \;(k = 1,\dots, n-2)
 \]
  form an orthonormal basis for ${\mathfrak m}_{\varepsilon/2}.$ From Lemma \ref{bracket} and (\ref{eq.ms4.3}), we have that the vectors
\[
 \zeta^{1}_{\varepsilon} = -\frac{1}{\langle\alpha_{1},\alpha_{1}\rangle} {\rm i}{t_{\alpha_{1}}}, \quad \zeta^{2+a}_{\varepsilon} = \frac{\sqrt{2}\beta}{2}(U^{a}_{\widetilde{\alpha_{22}}} - U^{a}_{\mu}),\;\;(a = 0,1)
 \]
 form an orthonormal basis of ${\mathfrak k}_{\varepsilon}$ and, of ${\mathfrak k}_{\varepsilon/2},$ the following set of vectors
 \[
 \zeta^{j,a}_{\varepsilon/2}  = \beta U^{a}_{\alpha_{2(j+1)}}, \quad \zeta^{n+k,a}_{\varepsilon/2} = \beta U^{a}_{\widetilde{\alpha_{2(k+2)}}}.
  \]
 Let $\Delta^{+}_{1}(t) = \{\alpha\in \Delta^{+}(t)\colon \alpha\pm \alpha_{1}\not\in\Delta\} = \{\alpha_{pq}\; (3\leq p\leq q \leq n +1),\; \widetilde{\alpha_{pq}}\; (3\leq p\leq q \leq n)\}.$ Then ${\mathfrak h} = {\mathfrak h}_{0} \oplus {\mathfrak h}_{1},$ where ${\mathfrak h}_{0} = {\mathfrak t}_{0}$ and ${\mathfrak h}_{1}$ is the subspace 
 \[
 {\mathfrak h}_{1} = {\mathbb R}\{ U^{0}_{\mu} + U^{0}_{\widetilde{\alpha_{22}}},U^{1}_{\mu} + U^{1}_{\widetilde{\alpha_{22}}}\}\oplus \sum_{\alpha\in \Delta^{+}_{1}(t)}
 ({\mathbb R} U^{0}_{\alpha} + {\mathbb R}U^{1}_{\alpha}).
 \] 
 Hence, ${\mathfrak h} = {\mathbb R}\{{\rm i} t_{\alpha_{1}} + 2{\rm  i} t_{\alpha_{2}}, U^{a}_{\mu} +U^{a}_{\widetilde{\alpha_{22}}};\;a = 0,1\}\oplus {\mathfrak s}{\mathfrak p}(n-1).$ By Lemma \ref{bracket} and since $\mu\pm \widetilde{\alpha_{22}}\notin\Delta,$  $\langle\mu,\widetilde{\alpha_{22}}\rangle = 0$ and $\langle\mu,\mu\rangle = \langle\widetilde{\alpha_{22}},\widetilde{\alpha_{22}}\rangle = 2\langle\alpha_{1},\alpha_{1}\rangle,$ ${\mathbb R}\{{\rm i}t_{\mu} + {\rm i}t_{\widetilde{\alpha_{22}}}, U^{a}_{\mu} + U^{a}_{\widetilde{\alpha_{22}}}\}$ is a subalgebra isomorphic to ${\mathfrak s}{\mathfrak p}(1),$ which using that $\langle\mu,\alpha\rangle = \langle\widetilde{\alpha_{22}},\alpha\rangle = 0,$ $\mu\pm\alpha\notin\Delta$ and $\widetilde{\alpha_{22}}\pm \alpha\notin\Delta$ for all $\alpha\in \Delta^{+}_{1}(t),$ it is the center ${\mathfrak z}({\mathfrak h})$ of ${\mathfrak h}.$ On the other hand, ${\rm i}t_{\mu} + {\rm i}t_{\widetilde{\alpha_{22}}} = 2(({\rm i}t_{\alpha_{1}} + 2{\rm i}t_{\alpha_{2}}) + 2({\rm i}t_{\alpha_{3}} + \dots + {\rm i}t_{\alpha_{n}}) + {\rm i}t_{\alpha_{n+1}}).$ This implies that ${\mathfrak h} = {\mathfrak s}{\mathfrak p}(1)\oplus {\mathfrak s}{\mathfrak p}(n-1).$ 
 
 Using (\ref{*}), (\ref{**}), (\ref{***}) and Lemma \ref{bracket}, we obtain
$$
 \begin{array}{l}
  [\xi^{1}_{\varepsilon},\xi^{l,a}_{\varepsilon/2}]  =  [\zeta^{1}_{\varepsilon},\zeta^{l,a}_{\varepsilon/2}] = \frac{(-1)^{a}}{2}\zeta^{l,a+1}_{\varepsilon/2}, \quad [\zeta^{1}_{\varepsilon}, \xi^{l,a}_{\varepsilon/2}] = -[\xi^{1}_{\varepsilon}, \zeta^{l,a}_{\varepsilon/2}] = \frac{(-1)^{a+1}}{2}\xi^{l,a+1}_{\varepsilon/2},\\[0.5pc] 
 
 [\xi^{2+b}_{\varepsilon},\xi^{k,a}_{\varepsilon/2}]  =  [\zeta^{2+b}_{\varepsilon},\zeta^{k,a}_{\varepsilon/2}] = \frac{(-1)^{c}}{2}\zeta^{n+k,a+b}_{\varepsilon/2}, \quad [\xi^{2+b}_{\varepsilon}, \xi^{n+k,a}_{\varepsilon/2}] = [\zeta^{2+b}_{\varepsilon}, \zeta^{n+k,a}_{\varepsilon/2}] = \frac{(-1)^{c+1}}{2}\zeta^{k,a+b}_{\varepsilon/2},\\[0.5pc]

 [\zeta^{2+b}_{\varepsilon/2},\xi^{k,a}_{\varepsilon/2}]  = - [\xi^{2+b}_{\varepsilon},\zeta^{k,a}_{\varepsilon/2}] = \frac{(-1)^{c+1}}{2}\xi^{n+k,a+b}_{\varepsilon/2}, \;\, [\zeta^{2+b}_{\varepsilon}, \xi^{n+k,a}_{\varepsilon/2}] = -[\xi^{2+b}_{\varepsilon}, \zeta^{n+k,a}_{\varepsilon/2}] = \frac{(-1)^{c}}{2}\xi^{k,a+b}_{\varepsilon/2},\\[0.5pc] 
 
 [\xi^{2+b}_{\varepsilon},\xi^{n-1,a}_{\varepsilon/2}] = [\zeta^{2+b}_{\varepsilon},\zeta^{n-1,a}_{\varepsilon/2}] = \frac{(-1)^{c}}{2}\zeta^{n,a+b}_{\varepsilon/2}, \quad [\xi^{2+b}_{\varepsilon},\xi^{n,a}_{\varepsilon/2}] = [\zeta^{2 +b}_{\varepsilon},\zeta^{n,a}_{\varepsilon/2}]= \frac{(-1)^{c+1}}{2}\zeta^{n-1,a+b}_{\varepsilon/2}\\[0.5pc]
 
[\zeta^{2+b}_{\varepsilon},\xi^{n-1,a}_{\varepsilon/2}] = -[\xi^{2+b}_{\varepsilon},\zeta^{n-1,a}_{\varepsilon/2}] = \frac{(-1)^{c+1}}{2}\xi^{n,a+b}_{\varepsilon}, \;\, [\zeta^{2+b}_{\varepsilon},\xi^{n,a}_{\varepsilon/2}] = -[\xi^{2+b}_{\varepsilon},\zeta^{n,a}_{\varepsilon/2}] = \frac{(-1)^{c}}{2}\xi^{n-1,a+b}_{\varepsilon/2},
  \end{array}
 $$ 
 where $1\leq l\leq 2n-2,$ $ 1\leq k\leq n-2,$ $ a,b = 0,1$ and $c = a(b+1).$
 
\noindent{\bf Case 4: ${\mathbb C}a{\mathbf P}^{2}.$} On ${\mathfrak f}_{4},$\vspace{-0.3cm}
$
 \xymatrix@R=.5cm@C=.8cm{
 \stackrel{2}{\stackrel{\circ}{\alpha_{1}}} \ar@1{-}[r] &
   \stackrel{3}{\stackrel{\circ}{\alpha_{2}}} \ar@2{->}[r]
   & \stackrel{4}{\stackrel{\circ}{\alpha_{3}}} \ar@1{-}[r] &
\stackrel{2}{\stackrel{\circ}{\alpha_{4}}}},
$ a set of positive roots is given by
$$
\begin{array}{lcl}
\Delta^{+} &\hspace{-0,3cm}  = \hspace{-0,3cm}&\{\alpha_{p,q}\;
 (1\leq p \leq q \leq 4),\; \alpha_{2,3} + \alpha_{3}, \alpha_{1,q} + \alpha_{p,3}\;(p=2,3;\;q=3,4),
 \alpha_{2,4} + \alpha_{3,q}\; (q=3,4), \\
  &\hspace{-0,6cm} & \alpha_{14} + \alpha_{p4}\;(p=2,3), \alpha_{1,4} + \alpha_{2,q} + \alpha_{3}\;(q = 3,4),\alpha_{1,4}+\alpha_{2,4}
  +\alpha_{p,3}+\alpha_{3}\;(p=2,3), \mu \},
\end{array}
$$
where $\mu= 2\alpha_{1,4} + \alpha_{2,3} + \alpha_{3}$ is the maximal root. We get the following relationships for the nonzero products on $\Pi=\{\alpha_{1},\alpha_{2},\alpha_{3},\alpha_{4}\}:$ 
\[
1 =  \langle\alpha_{3},\alpha_{3}\rangle = \langle\alpha_{4},\alpha_{4}\rangle  =\frac{1}{2} \langle\alpha_{1},\alpha_{1}\rangle = \frac{1}{2} \langle\alpha_{2},\alpha_{2}\rangle = -\langle\alpha_{1},\alpha_{2}\rangle =  -\langle\alpha_{2},\alpha_{3}\rangle = -2\langle\alpha_{3},\alpha_{4}\rangle.
\]
The Cayley plane ${\mathbb C}a{\mathbf P}^{2} = {\rm F}_{4}/{\rm Spin}(9)$ is obtained for $\sigma = {\rm Ad}_{\exp 2\pi{\rm i}t},$ where $t = t_{4}.$ Then $\Pi(t) = \{-\mu,\alpha_{1},\alpha_{2},\alpha_{3}\}.$ Putting $\beta_{1} = -\mu,$ $\beta_{2} = \alpha_{1},$ $\beta_{3} = \alpha_{2}$ and $\beta_{4} = \alpha_{3},$ we have
\[
\Delta^{+}(t) = \{\beta_{pq}\;(1\leq p\leq q\leq 4),\;\widetilde{\beta_{pq}} = \beta_{p} + \dots + 2\beta_{q} +\dots + 2\beta_{4}\; (1\leq p< q \leq 4)\},
\]
which coincides with a positive root set for ${\mathfrak b}_{4}:$
$
 \xymatrix@R=.5cm@C=.8cm{
 \stackrel{1}{\stackrel{\circ}{\beta_{1}}} \ar@1{-}[r] &
   \stackrel{2}{\stackrel{\circ}{\beta_{2}}} \ar@1{-}[r]
   & \stackrel{2}{\stackrel{\circ}{\beta_{3}}} \ar@2{->}[r] &
\stackrel{2}{\stackrel{\circ}{\beta_{4}}}},
$ and
\[
\Delta^{+}\setminus \Delta^{+}(t) = \{\alpha_{4},\alpha_{3,4},\alpha_{2,4},\alpha_{1,4}, \alpha_{1,4} + \alpha_{p,3}\;(p = 2,3), \alpha_{2,4} + \alpha_{3}, \alpha_{14} + \alpha_{2,3} + \alpha_{3}\}.
\]
We choose as Cartan subalgebra to ${\mathfrak a} = {\mathbb R}U^{0}_{\alpha_{4}}.$ Then, from Lemma \ref{bracket}, ${\mathfrak t}^{\mathbb C}_{\mathfrak a}$ is the complexification of ${\mathfrak t}_{\mathfrak a} = {\mathfrak a} \oplus {\mathfrak t}_{0},$ where ${\mathfrak t}_{0} = {\mathbb R}\{ {\rm i}t_{\alpha_{1}}, {\rm i}t_{\alpha_{2}}, {\rm i}t_{\alpha_{4}} + 2{\rm i}t_{\alpha_{3}}\}={\mathbb R}\{{\rm i}t_{\beta_{1}},{\rm i}t_{\beta_{2}},{\rm i}t_{\beta_{3}}\}.$ Note that $2(\alpha_{4} + 2\alpha_{3}) = -(\beta_{1} + 2\beta_{2} + 3\beta_{3}).$

Put $X = \frac{\sqrt{2}}{2}U^{0}_{\alpha_{4}}$ and take $\langle\cdot,\cdot\rangle= -{\mathbf B}.$ Then applying again Lemma \ref{bracket} and (\ref{*}), (\ref{**}) and (\ref{***}), the following basis is orthonormal for ${\mathfrak m}_{\varepsilon}$ and ${\mathfrak m}_{\varepsilon/2},$ respectively:
$$
\begin{array}{l}
\{\xi^{1}_{\varepsilon} = \frac{\sqrt{2}}{2}U^{1}_{\alpha_{4}},\xi^{p,a}_{\varepsilon} = \frac{\sqrt{2}}{2}U^{a}_{\alpha_{14} + \alpha_{p3}}\;(p=2,3), \xi^{4,a}_{\varepsilon} = \frac{\sqrt{2}}{2}U^{a}_{\alpha_{24} + \alpha_{3}}\}\\[0.4pc]
\{\xi^{q,a}_{\varepsilon/2} = \frac{\sqrt{2}}{2}U^{a}_{\alpha_{q4}}\;(q = 1,2,3),\; \xi^{4,a}_{\varepsilon/2} = \frac{\sqrt{2}}{2}U^{a}_{\alpha_{14} + \alpha_{23} + \alpha_{3}}\}.
\end{array}
$$
 The corresponding orthonormal basis for ${\mathfrak k}_{\varepsilon}$ is $\{\zeta^{1}_{\varepsilon},\zeta^{p,a}_{\varepsilon}\;(p=2,3), \zeta^{4,a}_{\varepsilon}\}$ and, for ${\mathfrak k}_{\varepsilon/2},$ $\{\zeta^{q,a}_{\varepsilon/2}\;(q = 1,2,3),\zeta^{4,a}_{\varepsilon/2}\},$ where
$$
\begin{array}{l}
\zeta^{1}_{\varepsilon} = -{\rm i}t_{\alpha_{4}} = \frac{1}{2}({\rm i}t_{\beta_{1}} + 2{\rm i}t_{\beta_{2}} + 3{\rm i}t_{\beta_{3}} + 4{\rm i}t_{\beta_{4}}),\\[0.4pc]
\zeta^{p,a}_{\varepsilon} = \frac{1}{2}(U^{a}_{\alpha_{13} + \alpha_{p3}}- U^{a}_{\alpha_{14} + \alpha_{p4}}) = \frac{1}{2}(U^{a}_{\widetilde{\beta_{2(p+1)}}} + (-1)^{a}U^{a}_{\widetilde{\beta_{1(6-p)}}}),\\[0.4pc]
\zeta^{4,a}_{\varepsilon} = \frac{1}{2}(U^{a}_{\alpha_{23} + \alpha_{3}} - U^{a}_{\alpha_{24} + \alpha_{34}}) = \frac{1}{2}(U^{a}_{\widetilde{\beta_{34}}} + (-1)^{a}U^{a}_{\widetilde{\beta_{12}}}),\\[0.4pc]
\zeta^{q,a}_{\varepsilon/2} = \frac{\sqrt{2}}{2}U^{a}_{\alpha_{q3}} =\frac{\sqrt{2}}{2}U^{a}_{\beta_{(q+1)4}},\quad \zeta^{4,a}_{\varepsilon/2} = \frac{\sqrt{2}}{2}U^{a}_{\alpha_{14} + \alpha_{24} + \alpha_{3}} = \frac{\sqrt{2}}{2}(-1)^{a+1}U^{a}_{\beta_{14}}.
\end{array}
$$

The subalgebra ${\mathfrak h}\subset {\mathfrak k}$ is of type ${\mathfrak b}_{3},$ given by ${\mathfrak h} = {\mathfrak h}_{0}\oplus {\mathfrak h}_{1},$ where ${\mathfrak h}_{0} = {\mathfrak t}_{0}$ and 
\[
{\mathfrak h}_{1}={\mathbb R}\{U^{a}_{\widetilde{\beta_{2(p+1)}}} + (-1)^{a+1}U^{a}_{\widetilde{\beta_{1(6-p)}}}\;(p=2,3), U^{a}_{\widetilde{\beta_{34}}} + (-1)^{a+1}U^{a}_{\widetilde{\beta_{12}}}, U^{a}_{\beta_{st}}\;\;(1\leq s\leq t\leq 3)\}.
\]
This proves Lemma \ref{restricted} and, by direct calculation, also Lemma \ref{pbrack} for this case.

\section{Tangent sphere bundles of compact rank-one symmetric spaces} 

\subsection{Tangent sphere bundles as homogeneous manifolds} Let $(M=G/K,g)$ be a compact rank-one Riemannian symmetric space and let ${\mathfrak a} = {\mathbb R}\{X\}$ be a Cartan subspace of ${\mathfrak m} = T_{o}M,$ where $X$ is the unique vector in ${\mathfrak a}$ such that ${\varepsilon}_{\mathbb R}(X) = 1$ and the $G$-invariant Riemannian metric $g$ is determined by the inner product $\langle\cdot,\cdot\rangle$ such that $\langle X,X\rangle = 1.$ Consider the Weyl chamber $W^{+}$ in ${\mathfrak a}$ containing $X$ given by
\[
W^{+} = \{ w\in {\mathfrak a}\colon \varepsilon_{\mathbb R}(w)>0\} = \{tX\;:\; t\in {\mathbb R}^{+}\},
\]
which is naturally identified with ${\mathbb R}^{+}.$ Since each nonzero ${\rm Ad}(K)$-orbit in ${\mathfrak m}$ intersects ${\mathfrak a}$ and also $W^{+},$ the open dense subset of regular points ${\mathfrak m}^{R}:= {\rm Ad}(K)(W^{+})$ of ${\mathfrak m}$  is ${\mathfrak m}\setminus \{0\}$ and $\phi(G\times_{K}{\mathfrak m}^{R})$ is the punctured tangent bundle $D^{+}(G/K).$ This means that the action of the subgroup ${\rm Ad}(K)$ on the unit sphere ${\mathcal S}_{\mathfrak m}(1)$ of ${\mathfrak m}$ is transitive. Such property characterizes the class of compact rank-one symmetric spaces. In fact, if ${\rm rank} \,G/K\geq 2,$ the hyperplane $\lambda(w) =0$ for each restricted root $\lambda$ intersects ${\mathcal S}_{\mathfrak m}(1)$ and it implies that the ${\rm Ad}(K)$-action cannot be transitive on ${\mathcal S}_{\mathfrak m}(1).$ 

Let $H$ be the closed subgroup of $K$ defined by
 \[
 H = \{k\in K\colon {\rm Ad}_{k} X = X\}.
 \] 
 Then $H = G_{\mathfrak a}\cap K,$ where $G_{\mathfrak a}$ is the centralizer of ${\mathfrak a},$ and ${\mathfrak h}$ is the Lie algebra of $H.$ Moreover, ${\mathcal S}_{\mathfrak m}(1),$ under the $K$-action determined by the isotropy representation, is expressed as the quotient manifold $K/H.$  
 
Consider the projection $\pi_{H}\times{\rm id}:G\times {\mathbb R}^{+}\to G/H\times {\mathbb R}^{+},$ $(a,t)\mapsto (aH,t),$ and the mapping
\begin{equation}\label{f+}
f^{+}\colon G/H\times {\mathbb R}^{+}\to G\times_{K}({\mathfrak m}\setminus\{0\}),\quad (aH,t)\mapsto [(a,tX)],
\end{equation}
which is a well-defined $G$-equivariant diffeomorphism and $(f^{+})^{-1}([(a,{\rm Ad}_{k}w)]) = (akH,t),$ for all $a\in G,$ $k\in K$ and $w= tX\in W^{+}.$ Then, $G/H\times {\mathbb R}^{+}$ is identified via $\phi\circ f^{+}$ with $D^{+}(G/K).$ 
\begin{proposition}\label{phomo} The natural action of $G$ on $T(G/K)\cong G\times_{K}{\mathfrak m}$ is of cohomogeneity one. The principal orbits are the tangent sphere bundles $T_{r}(G/K)$ of radius $r>0,$ which are diffeomorphic to the quotient manifold $G/H.$ The zero section is the unique singular orbit. Moreover, the natural projection $\pi^{T}\colon T_{r}(G/K)\to G/K$ determines the homogeneous fibration 
 \begin{equation}\label{fibration}
 {\mathcal S}_{\mathfrak m}(r) = K/H\to T_{r}(G/K) = G/H\to G/K,\quad aH \to aK.
 \end{equation}
 \end{proposition}
\begin{proof} From the transitivity of the isotropy action on ${\mathcal S}_{\mathfrak m}(1),$ for each $u\in {\mathfrak m}\setminus\{0\}$ there exists $k\in K$ such that $u = r{\rm Ad}_{k}X,$ where $\langle u,u\rangle = r^{2},$ and the $G$-orbit $G\cdot [(a,u)] = G\times_{K}\{u\}$ of $[(a,u)]$ in $G\times_{K}{\mathfrak m}$ can be expressed as 
\[
G\cdot [(a,u)] = \{[(a,r{\rm Ad}_{k}X)]\colon a\in G\} = G\times_{K}{\mathcal S}_{\mathfrak m}(r).
\]
Let $\iota_{r}\colon G/H\to G/H\times {\mathbb R}^{+}$ be the imbedding given by $\iota_{r}(aH) = (aH,r).$ Then $f^{+}\circ \iota_{r}$ is a $G$-equivariant diffeomorphism from $G/H$ onto $G\cdot [(a,u)].$ 

Taking into account that for each $(a,x)\in G\times{\mathfrak m},$ we get $\langle x,x\rangle = g_{\tau_{g}(o)}((\tau_{a})_{*o}x, (\tau_{a})_{*o}x) = g_{\tau_{a}(o)}(\phi[(a,x)],\phi[(a,x)]),$ the tangent sphere bundle $T_{r}(G/K)$ is given by
 \begin{equation}\label{TGK}
 T_{r}(G/K) = \phi(G\times_{K}{\mathcal S}_{\mathfrak m}(r)).
 \end{equation}
Hence, using that $\phi$ is $G$-equivariant, the $G$-orbit of $(\tau_{a})_{*o}u$ is $\phi(G\cdot[(a,u)] = T_{r}(G/K)$ and the orbit $\phi(G\cdot[(a,0)])$ of the zero vector in $T_{aK}(G/K)$ is clearly the zero section in $T(G/K).$ Therefore, the first part of the proposition is proved. For the second, note that by definition, $H$ is the isotropy subgroup of the ${\rm Ad}(K)$-action on ${\mathcal S}_{\mathfrak m}(r)$ at $rX.$
\end{proof}

The compact Lie groups acting effectively and transitively on some sphere ${\mathbb S}^{n}\subset {\mathbb R}^{n+1}$ have been classified by D. Montgomery and H. Samelson and A. Borel (see \cite{O} for more in\-for\-ma\-tion). They are given, together with their isotropy subgroups and isotropy representations in Table II, where the superscript indicates the dimension of the corresponding irreducible subspace.

\vspace{0.2cm}

 \begin{tabular}{llll}
\multicolumn{4}{c}{Table II. Compact transitive groups on spheres}\\[0.6pc]
\hline\noalign{\smallskip}
 $G$ & $K$ & $\dim G/K$ & \mbox{isotropy repr.}\\
\noalign{\smallskip}\hline\noalign{\smallskip}
$(1)\;\Lie{SO}(n)$ & $\Lie{SO}(n-1)$ & $n-1$ & \mbox{irred.}\\[0.3pc]
$(2)\;\Lie{SU}(n)$ & $\Lie{SU}(n-1)$ & $2n-1$ & ${\mathfrak m} = {\mathfrak m}_{1}^{1}\oplus {\mathfrak m}_{2}^{2n-2}$\\[0.3pc]
$(3)\;\Lie{U}(n)$ & $\Lie{U}(n-1)$ & $2n-1$ & ${\mathfrak m} = {\mathfrak m}_{1}^{1}\oplus {\mathfrak m}^{2n-2}_{2}$ \\[0.3pc]
$(4)\;\Lie{Sp}(n)$ & $\Lie{Sp}(n-1)$ & $4n-1$ & ${\mathfrak m} = {\mathfrak m}_{1}^{3}\oplus {\mathfrak m}_{2}^{4n-4}$ \\[0.3pc]
$(5)\;\Lie{Sp}(n)\times \Lie{Sp}(1)$ & $\Lie{Sp}(n-1)\times \Lie{Sp}(1)$ & $4n-1$ & ${\mathfrak m} = {\mathfrak m}_{1}^{3}\oplus {\mathfrak m}_{2}^{4n-4}$\\[0.3pc]
$(6)\;\Lie{Sp}(n)\times \Lie{U}(1)$ & $\Lie{Sp}(n-1)\times\Lie{U}(1)$ & $4n-1$ & ${\mathfrak m} = {\mathfrak m}_{1}^{1}\oplus{\mathfrak m}_{2}^{2}\oplus {\mathfrak m}_{3}^{4n-4}$\\[0.3pc]
$(7)\;\Lie{Spin}(9)$ & $\Lie{Spin}(7)$ & $15$ & ${\mathfrak m} = {\mathfrak m}_{1}^{7}\oplus {\mathfrak m}_{2}^{8}$\\[0.3pc]
$(8)\;\Lie{Spin}(7)$ & $\Lie{G}_{2}$ & $7$ & \mbox{irred.}\\[0.3pc]
$(9)\;\Lie{G}_{2}$ & $\Lie{SU}(3)$ & $6$ & \mbox{irred.}\\[0.3pc]
\hline
\end{tabular}

\vspace{0.1cm} 
\begin{proposition}\label{quotients} Tangent sphere bundles of compact rank-one symmetric spaces and their $\pi^{T}$-homogeneous fibrations admit the following quotient expressions:
 $$
 \begin{array}{l}
 {\mathbb S}^{n-1} = \SO(n)/\SO(n-1)\to T_{r}{\mathbb S}^{n} = \SO(n+1)/\SO(n-1)\to {\mathbb S}^{n},\\[0.4pc]
 {\mathbb S}^{n-1} = {\rm O}(n)/{\rm O}(n-1)\to T_{r}{\mathbb R}{\mathbf P}^{n} = \SO(n+1)/S({\rm O}(1)\times {\rm O}(n-1))\to {\mathbb R}{\mathbf P}^{n},\\[0.4pc]
 {\mathbb S}^{2n-1} = \SU(n)/\SU(n-1)\to T_{r}{\mathbb C}{\mathbf P}^{n} = \SU(n+1)/S(\Un(1)\times\Un(n-1))\to {\mathbb C}{\mathbf P}^{n},\\[0.4pc]
 {\mathbb S}^{4n-1}\hspace{-0.1cm} = \hspace{-0.1cm}\Sp(n)\times\Sp(1)/\hspace{-0.1cm}\Sp(n-1)\times\Sp(1)\to T_{r}{\mathbb H}{\mathbf P}^{n}\hspace{-0.1cm} =\hspace{-0.1cm} \Sp(n+1)/\hspace{-0.1cm}\Sp(n-1)\times\Sp(1)\to {\mathbb H}{\mathbf P}^{n},\\[0.4pc]
 {\mathbb S}^{15} = \Spin(9)/\Spin(7)\to T_{r}{\mathbb C}a{\mathbf P}^{2}= {\rm F}_{4}/\Spin(7)\to {\mathbb C}a{\mathbf P}^{2}.
 \end{array}
 $$
 \end{proposition}
 \begin{proof} Applying Proposition \ref{phomo}, Table II and the cases studied in Section \ref{subrestricted}, it is explicitly obtained the subgroup $H$ in each compact rank-one symmetric space:

\noindent {\bf Case 1: ${\mathbb S}^{n}$ and ${\mathbb R}{\mathbf P}^{n}$} $(n\geq 2).$ For ${\mathbb S}^{n} = G/K = \SO(n+1)/\SO(n),$ using (\ref{hh}), the identity connected component $H_{0}$ of $H$ is $I_{2}\times \SO(n-1)\subset \SO(n+1).$ Then, $K/H_{0} = \SO(n)/\SO(n-1)$ is the $(n-1)$-sphere and hence, $H = H_{0}.$
 
 For ${\mathbb R}{\mathbf P}^{n} = G/K =\SO(n+1)/S({\rm O}(1)\times {\rm O}(n)),$ taking again $X$ as $A_{12},$ we also have $H_{0} = I_{2}\times\SO(n-1).$ Since $K$ has two connected components: $1\times SO(n)$ and $-1\times O^{-}(n),$ where $O^{-}(n)$ denotes the matrices in $O(n)$ with determinant $-1,$ $H$ is not connected. Consider the natural isomorphism $\psi\colon {\rm O}(n)\to S({\rm O}(1)\times{\rm O}(n)),$ $\psi(A) = ({\rm det}\,A)\times A\in \SO(n+1),$ $A\in {\rm O}(n).$ Then, from (\ref{fibration}), ${\mathbb S}^{n-1} = K/H = {\rm O}(n)/{\rm O}(n-1).$ Hence, $H = \psi({\rm O}(n-1)) = S({\rm O}(1)\times {\rm O}(n-1))\subset S({\rm O}(1)\times {\rm O}(n)).$
 
 \vspace{0.1cm}
 
 \noindent {\bf Case 2:} ${\mathbb C}{\mathbf P}^{n}$ $(n\geq 2).$ Here, $K = S(\Un(1)\times \Un(n))\subset \SU(n+1).$ The center ${\mathfrak z}({\mathfrak k})$ of the Lie algebra ${\mathfrak k}$ of $K$ is one-dimensional and 
 \[
 Z_{0} = {\rm diag}({\rm i}b_{0},{\rm i}(b_{0}-1),\dots ,{\rm i}(b_{0}-1)),\quad b_{0} = \frac{n}{n+1},
 \]
 is a generator.
 From (\ref{4.27}) it is obtained that $H_{0} = \{\exp tZ_{1}\colon t\in {\mathbb R}\}\cdot SU(n-1),$ where $Z_{1}$ is given in (\ref{Z1}). Hence the matrix
\[
 {\rm diag}({\rm i}b_{1},{\rm i}b_{1},{\rm i}(b_{1}-1),\dots ,{\rm i}(b_{1}-1)),\quad b_{1} = \frac{n-1}{n+1},
 \]
 is a generator of ${\mathfrak z}({\mathfrak h})$ and
 \[
 H_{0} = \{e^{i tb_{1}}I_{2}\times e^{i t(b_{1}-1)}A\in \SU(n+1)\colon  t\in {\mathbb R},\; A\in \SU(n-1)\}.
 \]
 Because $\widetilde{Z_{0}} := {\rm diag}(0,{\rm i}\widetilde{b_{0}},{\rm i}(\widetilde{b_{0}}-1),\dots , {\rm i}(\widetilde{b_{0}}-1)),$ with $\widetilde{b_{0}} = \frac{b_{1}}{b_{0}} = \frac{n-1}{n},$ belongs to the center of the Lie algebra of $1\times{\rm S}(\Un(1)\times \Un(n-1))\subset \SU(n+1),$ $H_{0}$ can be identified with this subgroup via the isomorphism $(\exp tZ_{1})A\mapsto (\exp t \widetilde{Z_{0}})A,$ $A\in \SU(n-1),$ or equivalently, with ${\rm S}(\Un(1)\times \Un(n-1))$ via the isomorphism
 \[
 e^{i tb_{1}}I_{2}\times e^{i t(b_{1}-1)}A \mapsto e^{it\widetilde{b_{0}}}\times 1 \times e^{it(\widetilde{b_{0}} -1)}A,\quad t\in {\mathbb R},\; A\in \SU(n-1).
 \]
Then $K/H_{0} = \SU(n)/\SU(n-1) = {\mathbb S}^{2n-1}$ and $H = H_{0}.$
 
 \vspace{0.1cm}
 
 \noindent {\bf Case 3:} ${\mathbb H}{\mathbf P}^{n}$  $(n\geq 1).$  According with Section \ref{subrestricted}, the centers of ${\mathfrak h}$ and of ${\mathfrak k}$ are isomorphic to ${\mathfrak s}{\mathfrak p}(1).$ They are given by 
\[ 
 {\mathfrak z}({\mathfrak h}) = {\mathbb R}\{{\rm i}t_{\mu} + {\rm i}t_{\widetilde{\alpha_{22}}}, U^{a}_{\mu} + U^{a}_{\widetilde{\alpha_{22}}}\colon a = 0,1\},\quad  {\mathfrak z}({\mathfrak k}) = {\mathbb R}\{{\rm i}t_{\mu}, U^{a}_{\mu}\colon a = 0,1\}.
 \]
 Hence, $H_{0} = \Sp(1)\cdot\Sp(n-1)$ and $K/H_{0} = \Sp(n)\times\Sp(1)/\Sp(n-1)\times \Sp(1) = {\mathbb S}^{4n-1}.$ This means that $H= H_{0}.$ 
 
 \vspace{0.1cm}
 
 \noindent {\bf Case 4:} ${\mathbb C}a{\mathbf P}^{2}.$ In Section \ref{subrestricted} it has been proven that ${\mathfrak h} = {\mathfrak s}{\mathfrak o}(7).$ Since $K = \Spin(9),$ $H_{0} = \Spin(7).$ Then, $K/H_{0} = {\mathbb S}^{15}$ and so, $H = H_{0}.$ 
 \end{proof}
 
 Let ${\mathbb K}$ be the field of real numbers, complex numbers or quaternions. On ${\mathbb K}^{n+1},$ $n\geq 0,$ consider its standard inner product. The {\em Stiefel manifold} $V_{2}({\mathbb K}^{n+1})$ of all orthonormal $2$- frames in ${\mathbb K}^{n+1}$ is a homogeneous manifold with quotient expression: $\SO(n+1)/\SO(n-1),$ $\SU(n+1)/\SU(n-1)$ or $\Sp(n+1)/\Sp(n-1)$ depending on whether ${\mathbb K} = {\mathbb R},$ ${\mathbb C}$ or ${\mathbb H},$ and the {\em projective Stiefel manifold} $W_{2}({\mathbb K}^{n+1})$ is the quotient space of the free $S$-action on $V_{2}({\mathbb K}^{n+1}),$ where $S$ is the subset of unit vectors of ${\mathbb K}$ $(S = {\mathbb S}^{0},$ ${\mathbb S}^{1}$ or ${\mathbb S}^{3}).$ 
 
 Clearly, the tangent sphere bundle $T_{r}{\mathbb S}^{n}$ of ${\mathbb S}^{n}\subset {\mathbb R}^{n+1}$ is naturally identified with $V_{2}({\mathbb R}^{n+1}),$ by considering each $u = u_{x}\in {T_{r}}_{x}{\mathbb S}^{n}$ as the pair $(x,\frac{1}{r}u)$ in ${\mathbb R}^{n+1},$ and $T_{r}{\mathbb R}{\mathbf P}^{n}$ with $W_{2}({\mathbb R}^{n+1}),$ by identification of $u$ with its opposite. As a direct consequence from Proposition \ref{quotients}, we have the following.
\begin{corollary} $T_{r}{\mathbb K}{\mathbf P}^{n} = W_{2}({\mathbb K}^{n+1}),$ for ${\mathbb K} = {\mathbb R},$ ${\mathbb C}$ or ${\mathbb H}.$ 
\end{corollary}
\subsection{The set of all invariant metrics on $T_{r}(G/K)$} According with Proposition \ref{phomo}, the tangent sphere bundle $T_{r}(G/K)$ can be expressed as the homogeneous manifold $G/H.$ Let $\pi_{H}\colon G\to G/H$ be the natural projection and put $\overline{\mathfrak m} = {\mathfrak m}\oplus {\mathfrak k}^{+}.$ Then ${\mathfrak g} = \overline{\mathfrak m}\oplus {\mathfrak h}$ is a reductive decomposition associated to $G/H$ and $T_{o_{H}}G/H$, $o_{H}= \{H\}$ being the origin of $G/H,$ is identified with $\overline{\mathfrak m}$ via $(\pi_{H})_{*e}.$ As in Section \ref{subrestricted}, we choose the ${\rm Ad}(G)$-invariant inner product $\langle\cdot,\cdot\rangle$ on ${\mathfrak g}$ satisfying $\langle X,X\rangle = 1,$ where $X\in {\mathfrak a}$ and $\varepsilon_{\mathbb R}(X) = 1.$ Next, using the identification $\overline{\mathfrak m}\cong T_{o_{H}}(G/H),$ all the $G$-invariant metrics $\tilde{\mathbf g}$ on $T_{r}(G/K) = G/H$ are obtained.
\begin{proposition}\label{metricB} Any $G$-invariant metric on $T_{r}(G/K)$ is determined by an ${\rm Ad}(H)$-invariant inner product $\langle\cdot,\cdot\rangle^{(a;a_{\lambda};b_{\lambda})},$ $\lambda\in \Sigma^{+},$ on $\overline{\mathfrak m}$ of the form 
\begin{equation}\label{Bb}
\langle\cdot,\cdot\rangle^{(a,a_{\varepsilon},a_{\varepsilon/2},b_{\varepsilon},b_{\varepsilon/2})} = a^{2}\langle\cdot,\cdot\rangle_{\mathfrak a} + a_{\varepsilon}\langle\cdot,\cdot\rangle_{{\mathfrak m}_{\varepsilon}}  + a_{\varepsilon/2}\langle\cdot,\cdot\rangle_{{\mathfrak m}_{\varepsilon/2}} + b_{\varepsilon}\langle\cdot,\cdot\rangle_{{\mathfrak k}_{\varepsilon}} + b_{\varepsilon/2}\langle\cdot,\cdot\rangle_{{\mathfrak k}_{\varepsilon/2}},
\end{equation}
where $a,a_{\varepsilon},  a_{\varepsilon/2},b_{\varepsilon},$ and $b_{\varepsilon/2}$ are positive constants.
\end{proposition} 
\begin{proof}
 By definition of $H,$ the vector $X\in {\mathfrak a}$ is ${\rm Ad}(H)$-invariant and we get ${\rm ad}^{2}_{X}\circ {\rm Ad}_{h} = {\rm Ad}_{h}\circ {\rm ad}^{2}_{X},$ for each $h\in H.$ Hence the subspaces ${\mathfrak m}_{\varepsilon}, {\mathfrak m}_{\varepsilon/2}, {\mathfrak k}_{\varepsilon}$ and ${\mathfrak k}_{\varepsilon/2}$ are all ${\rm Ad}(H)$-invariant. Now, using the isotropy representation of spheres in Table II, it follows from Proposition \ref{quotients} that the subspaces ${\mathfrak k}_{\varepsilon}$ and ${\mathfrak k}_{\varepsilon/2}$ of ${\mathfrak k}$ are ${\rm Ad}(H)$-irreducible. Hence, applying (\ref{eq.ms4.3}), the representation ${\rm Ad}(H)$ of $H$ is also irreducible on ${\mathfrak m}_{\varepsilon}$ and ${\mathfrak m}_{\varepsilon/2}.$ This proves the result.
 \end{proof}
 
In what follows we denote by $\tilde{\mathbf g}^{(a;a_{\lambda};b_{\lambda})}$ the $G$-invariant metric $\tilde{\mathbf g}$ in (\ref{Bb}). If $\varepsilon/2$ is not in $\Sigma,$ the numbers $a_{\varepsilon/2}$ and $b_{\varepsilon/2}$ can take any value and we will simply write $\tilde{\mathbf g} = \tilde{\mathbf g}^{(a,a_{\varepsilon},b_{\varepsilon})}.$ 

Let ${\mathfrak U}\colon \overline{\mathfrak m}\times \overline{\mathfrak m}\to \overline{\mathfrak m}$ be the symmetric bilinear function defined in (\ref{U}) for the Euclidean space $(\overline{\mathfrak m},\langle\cdot,\cdot\rangle^{(a;a_{\lambda};b_{\lambda})}).$ The following lemma is obtained using (\ref{eq.ms4.3}), (\ref{brack1}) and (\ref{Bb}).
\begin{lemma}\label{lU} We have:
$$
\begin{array}{l}
{\mathfrak U}({\mathfrak a},{\mathfrak a}) = {\mathfrak U}({\mathfrak m}_{\varepsilon},{\mathfrak m}_{\varepsilon}) = {\mathfrak U}({\mathfrak k}_{\varepsilon},{\mathfrak k}_{\varepsilon}) = 0,\\[0.4pc]
{\mathfrak U}(X,\xi^{j}_{\varepsilon}) = \frac{a^{2}-a_{\varepsilon}}{2b_{\varepsilon}}\zeta^{j}_{\varepsilon},\quad {\mathfrak U}(X,\zeta^{j}_{\varepsilon}) = \frac{b_{\varepsilon}-a^{2}}{2a_{\varepsilon}}\xi^{j}_{\varepsilon},\quad
{\mathfrak U}(X,\xi^{p}_{\varepsilon/2}) = \frac{a^{2}-a_{\varepsilon/2}}{4b_{\varepsilon/2}}\zeta^{p}_{\varepsilon/2},\\[0.4pc] 
{\mathfrak U}(X,\zeta^{p}_{\varepsilon/2}) = \frac{b_{\varepsilon/2}-a^{2}}{4a_{\varepsilon/2}}\xi^{p}_{\varepsilon/2}, \quad {\mathfrak U}(\xi^{j}_{\varepsilon},\zeta^{k}_{\varepsilon}) = \frac{a_{\varepsilon}-b_{\varepsilon}}{2a}\delta_{jk}X,\quad {\mathfrak U}(\xi^{j}_{\varepsilon},\xi^{p}_{\varepsilon/2}) = \frac{a_{\varepsilon/2}-a_{\varepsilon}}{2b_{\varepsilon/2}}[\xi^{j}_{\varepsilon},\xi^{p}_{\varepsilon/2}],\\[0.4pc]
{\mathfrak U}(\xi^{j}_{\varepsilon},\zeta^{p}_{\varepsilon/2}) = \frac{b_{\varepsilon/2}-a_{\varepsilon}}{2a_{\varepsilon/2}}[\xi^{j}_{\varepsilon},\zeta^{p}_{\varepsilon/2}],\quad {\mathfrak U}(\xi^{p}_{\varepsilon/2},\zeta^{j}_{\varepsilon}) = \frac{b_{\varepsilon}-a_{\varepsilon/2}}{2a_{\varepsilon/2}}[\xi^{p}_{\varepsilon/2},\zeta^{j}_{\varepsilon}],\\[0.4pc]
{\mathfrak U}(\zeta^{j}_{\varepsilon},\zeta^{p}_{\varepsilon/2}) = \frac{b_{\varepsilon/2}-b_{\varepsilon}}{2b_{\varepsilon/2}}[\zeta^{j}_{\varepsilon},\zeta^{p}_{\varepsilon/2}],\quad {\mathfrak U}(\xi^{p}_{\varepsilon/2},\zeta^{q}_{\varepsilon/2}) = \frac{a_{\varepsilon/2}-b_{\varepsilon/2}}{2}\Big(\frac{\delta_{pq}}{2a}X - \frac{1}{a_{\varepsilon}}[\xi^{p}_{\varepsilon/2},\zeta^{q}_{\varepsilon/2}]_{{\mathfrak m}_{\varepsilon}}\Big),
\end{array}
$$
for all $j,k = 1,\dots, m_{\varepsilon}$ and $p,q= 1,\dots ,m_{\varepsilon/2}.$
\end{lemma}

Since $X$ is ${\rm Ad}(H)$-invariant, it determines a $G$-invariant vector field on $T_{r}(G/K)$ that we also denote by $X.$ As a direct consequence from this lemma, the following is proved.
\begin{proposition}\label{pproperties} We have:
\begin{enumerate}
\item[{\rm (i)}] $X$ on $(T_{r}(G/K),\tilde{\mathbf g})$ is a Killing vector field if and only if $a_{\lambda} = b_{\lambda},$ for all $\lambda\in \Sigma^{+}.$
\item[{\rm (ii)}] $(T_{r}(G/K) =G/H,\tilde{\mathbf g})$ is naturally reductive {\rm (}with associated reductive decomposition ${\mathfrak g} = \overline{\mathfrak m}\oplus {\mathfrak h})$ if and only if $\tilde{\mathbf g}$ is proportional to $\langle\cdot,\cdot\rangle_{\overline{\mathfrak m}}.$
\item[{\rm (iii)}] The projection $\pi^{T}\colon (T_{r}(G/K),\tilde{\mathbf g})\to (G/K,g)$ is a Riemannian submersion if and only if $a = a_{\varepsilon} = a_{\varepsilon/2} = 1.$
\end{enumerate}
\end{proposition}

 \subsection{The standard almost contact metric structure} An affine connection $\nabla$ on ar\-bi\-trary smooth manifold $M$ defines a distribution ${\mathcal H}\colon u\in TM\to {\mathcal H}_{u}\subset T_{u}TM,$ called the {\em horizontal distribution}, where ${\mathcal H}_{u}$ is the space of all horizontal lifts of tangent vectors in $T_{p}M,$ $p = \pi^{T}(u),$ obtained by parallel translation with respect to $\nabla.$ (See, for example \cite[Ch. 9]{Bl} and references inside for more information). Then $\nabla$ induces the direct decomposition $T_{u}TM = {\mathcal H}_{u}\oplus {\mathcal V}_{u},$ where ${\mathcal V}_{u}$ is the vertical space ${\mathcal V}_{u} = {\rm Ker}(\pi^{T})_{*u}.$ The standard almost complex structure $J^{\nabla}$ on $TM$ induced by $\nabla,$ is defined by
\begin{equation}\label{J}
J^{\nabla}X^{\tt h} = X^{\tt v},\quad J^{\nabla} X^{\tt v} = -X^{\tt h},\quad X\in {\mathfrak X}(M),
\end{equation}
where $X^{\tt h}$ and $X^{\tt v}$ are the horizontal and vertical lift of $X,$ respectively. Given a Riemannian metric $g$ on $M,$ the Sasaki metric $g^{S}$ with respect to the pair $(g,\nabla)$ is given by
\begin{equation}\label{Sasaki}
g^{S}(\eta_{1},\eta_{2}) = (g((\pi^{T})_{*}\eta_{1},(\pi^{T})_{*}\eta_{2}) + g(K(\eta_{1}),K(\eta_{2})))\circ \pi^{T},\quad \eta_{1},\eta_{2}\in TT(M),
\end{equation}
where $K$ is the connection map of $\nabla.$ Then $g^{S}(X^{\tt h},Y^{\tt h}) = g^{S}(X^{\tt v},Y^{\tt v}) = g(X,Y)\circ \pi^{T}$ and $g^{S}(X^{\tt v},Y^{\tt h}) = 0,$ for all $X,Y\in {\mathfrak X}(M).$ This implies that $g^{S}$ is a Hermitian metric with respect to $J^{\nabla}.$ If $\nabla$ is the Levi-Civita connection of $(M,g),$ the pair $(J,g^{S}),$ where $J = J^{\nabla},$ known as the {\em standard almost Hermitian structure}, is almost K\"ahler.
 
 Now suppose that $M$ is a compact rank-one symmetric space $G/K,$ $g$ is the $G$-invariant Riemannian metric determined at the origin by the inner product $\langle\cdot,\cdot\rangle$ on ${\mathfrak m}$ and $\nabla$ is its Levi-Civita connection. Then $(J,g^{S})$ is defined on $G/H\times{\mathbb R}^{+}$ by its identification with $D^{+}(G/K)$ via $\phi\circ f^{+},$ i.e. $J = (\phi\circ f^{+})^{-1}_{*}J(\phi\circ f^{+})_{*}$ and $g^{S} = (\phi\circ f^{+})^{*}g^{S}.$ Moreover, the tangent space $T_{(o_{H},t)}(G/H\times {\mathbb R}^{+})$ is taken as $\overline{\mathfrak m}\times T_{t}{\mathbb R},$ via $(\pi_{H}\times {\rm id})_{*(e,t)},$ for each $t\in {\mathbb R}^{+},$ and the vectors
\begin{equation}\label{basis}
\{(X,0),(0,\frac{\partial}{\partial t}),(\xi^{j}_{\varepsilon},0),(\xi^{p}_{\varepsilon/2},0),(\zeta^{j}_{\varepsilon},0),(\zeta^{p}_{\varepsilon/2},0);\; j = 1,\dots,m_{\varepsilon},\; p = 1,\dots,m_{\varepsilon/2}\}
\end{equation}
 is a basis of $T_{(o_{H},t)}(G/H\times {\mathbb R}^{+}).$ From (\ref{f+}), using (\ref{kk}), we have
\[
(f^{+})_{*(o_{H},t)}(\xi,\frac{\partial}{\partial t}) = \pi_{*(e,tX)}(\xi,X_{tX}) = \pi_{*(e,tX)}(\xi_{\mathfrak m}, X_{tX} + t[\xi_{\mathfrak k},X]_{tX}),
\]
for all $\xi\in \overline{\mathfrak m}.$ Then from (\ref{eq.ms4.3}), for all $s=1,\dots, m_{\lambda},$ $\lambda\in \Sigma^{+},$ we get
 \begin{equation}\label{f++}
\begin{array}{l}
(f^{+})_{*(o_{H},t)}(X,0)  =  \pi_{*(e,tX)}(X, 0),\;\;(f^{+})_{*(o_{H},t)}(0, \frac{\partial}{\partial t})  = \pi_{*(e,tX)}(0,X_{tX}),\\[0.4pc]
(f^{+})_{*(o_{H},t)}(\xi^{s}_{\lambda},0)  =  \pi_{*(e,tX)}(\xi^{s}_{\lambda}, 0),\;\; (f^{+})_{*(o_{H},t)}(\zeta^{s}_{\lambda},0)  =  \pi_{*(e,tX)}(0,-\lambda_{\mathbb R}(t)(\xi^{s}_{\lambda})_{tX}),
\end{array}
\end{equation}
where $\lambda_{\mathbb R}$ is taken as the mapping $\lambda_{\mathbb R}\colon {\mathbb R}^{+}\to {\mathbb R}^{+},$ $\lambda_{\mathbb R}(t) = \lambda_{\mathbb R}(tX).$ (Then, $\varepsilon_{\mathbb R} = {\rm id}_{\mid{\mathbb R}^{+}}.)$
\begin{proposition}\label{estandard} The standard almost Hermitian structure $(J,g^{S})$ on $G/H\times {\mathbb R}^{+}$ is $G$-invariant and it is determined at $(o_{H},t),$ $t\in {\mathbb R}^{+},$ by 
$$
\begin{array}{l}
J_{(o_{H},t)}(X,0)  =  (0, \frac{\partial}{\partial t}), \qquad J_{(o_{H},t)}(0,\frac{\partial}{\partial t}) = (-X, 0),\\[0.4pc]
J_{(o_{H},t)}(\xi^{s}_{\lambda},0) =  (-\frac{1}{\lambda_{\mathbb R}(t)}\zeta_{\lambda}^{s},0), \quad  J_{(o_{H},t)}(\zeta^{s}_{\lambda},0)  =  (\lambda_{\mathbb R}(t)\xi^{s}_{\lambda},0);\\[0.4pc]
g^{S}_{(o_{H},t)}((X,0),(X,0))  =  g^{S}_{(o_{H},t)}((0,\frac{\partial}{\partial t}),(0,\frac{\partial}{\partial t}) ) =  1,\\[0.4pc]
g^{S}_{(o_{H},t)}((\xi^{s}_{\lambda},0),(\xi^{s}_{\lambda},0))  =  \frac{1}{\lambda^{2}_{\mathbb R}(t)}g^{S}_{(o_{H},t)}((\zeta^{s}_{\lambda},0),(\zeta^{s}_{\lambda},0))  =  1,
\end{array}
$$
where $s = 1,\dots, m_{\lambda},$ $\lambda\in \Sigma^{+},$ being zero the rest of components of $g^{S}.$
\end{proposition}
\begin{proof} Denote by $\widetilde{\pi}$ the projection $\widetilde{\pi}\colon G\times_{K}{\mathfrak m}\to G/K.$ Then, $\widetilde{\pi} = \pi^{T}\circ \phi.$ Let $\widetilde{\mathcal V}$ and $\widetilde{\mathcal H}$ be the distributions $\widetilde{\mathcal V}:=(\phi^{-1})_{*}({\mathcal V}) = {\rm Ker}\;\widetilde{\pi}_{*}$ and $\widetilde{\mathcal H} :=(\phi^{-1})_{*}({\mathcal H})$ in $T(G\times_{K}{\mathfrak m}).$ Taking into account that $(\widetilde{\pi}\circ \pi)(a,x) = \pi_{K}(a),$ for all $(a,x)\in G\times {\mathfrak m},$ and using (\ref{lGm}), we have
\begin{equation}\label{vertical}
\widetilde{\mathcal V}_{[(a,x)]} = \{\pi_{*(a,x)}(0,u_{x})\;:\; u\in {\mathfrak m}\},\quad (a,x)\in G\times {\mathfrak m}.
\end{equation}
Because $\nabla$ is $G$-invariant, ${\mathcal V},$ $\widetilde{\mathcal V},$ ${\mathcal H}$ and $\widetilde{\mathcal H}$ are $G$-invariant distributions. Moreover,
\begin{equation}\label{verticallift}
\xi^{\tt v}_{[(e,x)]} = \pi_{*(e,x)}(0,\xi_{x}),\quad \xi\in {\mathfrak m}.
\end{equation}
For each $\eta\in T_{u}TM$ denote by $\eta^{\tt ver}$ (resp., $\eta^{\tt hor})$ the ${\mathcal V}_{u}$-component (resp., ${\mathcal H}_{u}$-component) of $\eta$ with respect to the decomposition $T_{u}T(G/K) = {\mathcal H}_{u}\oplus {\mathcal V}_{u}.$ Recall that the connection map $K$ is defined by $K_{u}(\eta) = \iota_{u}(\eta^{\tt ver}),$ where $\iota_{u}$ is the projection $\iota_{u}\colon T_{u}T(G/K)\to T_{p}(G/K)$ such that $\iota_{u}(\eta) = 0,$ for all $\eta\in {\mathcal H}_{u}$ and $\iota_{u}(v^{\tt v}_{u}) = v.$ This satisfies $K_{X_{p}}(X_{*p}u) = \nabla_{u}X,$ for all $X\in {\mathfrak X}(G/K)$ and $u\in T_{p}(G/K).$

For each $x\in {\mathfrak m},$ consider $x^{\tau}$ as a local section of $T(G/K).$ Then
\begin{equation}\label{xi+}
(x^{\tau})_{*o}\xi = \frac{d}{dt}_{\mid t = 0}x^{\tau}(\exp t\xi) = \frac{d}{dt}_{\mid t = 0}(\tau_{\exp t\xi_{*o}})x = \phi_{*[(e,x)]}(\pi_{*(e,x)}(\xi,0)),\quad \xi\in {\mathfrak m}.
\end{equation}
Since, using (\ref{nabla}), $K_{x}((x^{\tau})_{*o}\xi) = \nabla_{\xi}x^{\tau} = 0,$ we have $((x^{\tau})_{*o}\xi)^{\tt ver} = 0_{x}\in T_{x}T_{o}M.$ Then from (\ref{xi+}), the vertical component $(\pi_{*(e,x)}(\xi,0))^{\tt ver}$ of $\pi_{*(e,x)}(\xi,0)$ in the decomposition $T_{[(e,x)]}G\times_{K}{\mathfrak m} = \widetilde{\mathcal H}_{[(e,x)]}\oplus\widetilde{\mathcal V}_{[(e,x)]}$ is the zero vector. Hence, applying (\ref{vertical}),  
\[
(\pi_{*(e,x)}(\xi,u_{x}))^{\tt ver} = (\pi_{*(e,x)}(0,u_{x}))^{\tt ver} = \pi_{*(e,x)}(0,u_{x}),\;\;\mbox{\rm for all}\; (\xi,u)\in {\mathfrak m}\times{\mathfrak m}.
\]
Because $\widetilde{\pi}_{*[(e,x)]}(\pi_{*(e,x)}(\xi,u_{x})) = \xi\in {\mathfrak m}\cong T_{o}M,$ 
\begin{equation}\label{vh}
\xi^{\tt h}_{[(e,x)]} = (\pi_{*(e,x)}(\xi,u_{x}))^{\tt hor}_{[(e,x)]} = \pi_{*(e,x)}(\xi, 0_{x})
\end{equation}
and $\widetilde{\mathcal H}_{[(a,x)]} = \{\pi_{*(a,x)}(\xi^{\tt l}_{a},0_{x})\colon\xi\in {\mathfrak m}\}.$ Then, applying in (\ref{J}) and (\ref{Sasaki}), (\ref{verticallift}) and (\ref{vh}), $(J,g^{S})$ on $G\times_{K}{\mathfrak m}$ is determined by
\begin{equation}\label{JG}
\begin{array}{l}
J_{[(e,x)]}\pi_{*(e,x)}(\xi,u_{x}) = \pi_{*(e,x)}(-u,\xi_{x}),\\[0.4pc]
g^{S}(\pi_{*(e,x)}(\xi,u_{x}),\pi_{*(e,x)}(\eta,v_{x})) = \langle\xi,\eta\rangle + \langle u, v\rangle,
\end{array}
\end{equation}
for all $x,\xi, \eta, u, v\in {\mathfrak m}.$  Hence the result is proved using (\ref{f++}).
\end{proof}

  From Proposition \ref{estandard}, $(0,\frac{\partial}{\partial t})$ is a $G$-invariant unit vector field on the Riemannian manifold $(G/H\times {\mathbb R}^{+},g^{S})$ and it is normal to the submanifolds $\iota_{r}(G/H) = G/H\times \{r\}$ of $G/H\times {\mathbb R}^{+},$ for each $r>0.$ Hence we have the following.
\begin{proposition}\label{pstandard} The standard almost contact metric structure $(\varphi,\xi,\eta,\tilde{g}^{S} = \iota^{*}_{r}g^{S})$ on the tangent sphere bundle $T_{r}(G/K) = (\phi\circ f^{+})(G/H\times\{r\}),$ for each $r>0,$ is $G$-invariant and it is determined at $o_{H}$ by $\xi_{o_{H}}= X,$ $\eta_{o_{H}} = \langle X,\cdot\rangle,$ $\varphi_{o_{H}}X = 0$ and
 $$
\begin{array}{l}
\varphi_{o_{H}}\xi^{s}_{\lambda}  =  -\frac{1}{\lambda_{\mathbb R}(r)}\zeta^{s}_{\lambda}, \quad \varphi_{o_{H}}\zeta^{s}_{\lambda} = \lambda_{\mathbb R}(r)\xi^{s}_{\lambda}, \hspace{2.35cm} s = 1,\dots, m_{\lambda},\;\lambda\in \Sigma^{+};\\[0.4pc]
\tilde{g}^{S}_{o_{H}}= \left\{\begin{array}{lcl}\tilde{\mathbf g}_{o_{H}}^{(1,1,r^{2})} = \langle\cdot,\cdot\rangle_{\mathfrak m} + r^{2}\langle\cdot,\cdot\rangle_{{\mathfrak k}_{\varepsilon}} &\mbox{if} & G/K = {\mathbb S}^{n}\;\mbox{or}\;\; {\mathbb R}{\mathbf P}^{n},\\[0.4pc]
 \tilde{\mathbf g}^{(1;1;\lambda^{2}_{\mathbb R}(r))}_{o_{H}} =\langle\cdot,\cdot\rangle_{\mathfrak m} + r^{2}\langle\cdot,\cdot\rangle_{{\mathfrak k}_{\varepsilon}} + \frac{r^{2}}{4}\langle\cdot,\cdot\rangle_{{\mathfrak k}_{\varepsilon/2}} & \mbox{if} & G/K =  {\mathbb C}{\mathbf P}^{n},\; {\mathbb H}{\mathbf P}^{n}\;\mbox{or}\; \;{\mathbb C}a{\mathbf P}^{n}.
 \end{array}
 \right.
 \end{array}
 $$
 Moreover, the standard field $\xi$ of $T(G/K)$ is the $G$-invariant vector field determined by $\xi_{(o_{H},t)}= (X,0),$ for all $t\in {\mathbb R}^{+}.$
 \end{proposition}
 From Propositions \ref{pproperties} and \ref{pstandard}, the projection $\pi^{T}\colon (T_{r}(G/K),\widetilde{\mathbf g})\to (G/K,{\mathbf g})$ is a Riemannian submersion and the following is proved.
 \begin{corollary} The standard vector field $\xi$ of $(T_{r}(G/K),\tilde{g}^{S})$ is Killing if and only if $r = 1$ and $G/K = {\mathbb S}^{n}$ or ${\mathbb R}{\mathbf P}^{n}.$ Then $(G/H,\tilde{g}^{S})$ is naturally reductive.
 \end{corollary}
 
 \section{Invariant Sasakian structures on $T_{r}(G/K)$}  We first consider all $G$-invariant Riemannian metrics ${\mathbf g}$ on $G/H\times {\mathbb R}^{+}\cong D^{+}(G/K)$ such that the vector field $(0,\frac{\partial}{\partial t})$ on $G/H\times {\mathbb R}^{+}$ is orthogonal to the hypersurface $\iota_{r}(G/H)= G/H\times\{r\}\cong T_{r}(G/K),$ for each $r>0.$ From Proposition \ref{metricB}, ${\mathbf g}$ is determined at the points $(o_{H},t)\in G/H\times {\mathbb R}^{+}$ and with respect to the basis given in (\ref{basis}), by
\begin{equation}\label{gg}
\begin{array}{lcl}
{\mathbf g}_{(o_{H},t)}((X,0),(X,0)) = a^{2}(t), & {\mathbf g}_{(o_{H},t)}((0,\frac{\partial}{\partial t}),(0,\frac{\partial}{\partial t})) = b^{2}(t),\\[0.4pc]
{\mathbf g}_{(o_{H},t)}((\xi^{s}_{\lambda},0),(\xi^{s}_{\lambda},0)) = a_{\lambda}(t), & {\mathbf g}_{(o_{H},t)}((\zeta^{s}_{\lambda},0),(\zeta^{s}_{\lambda},0)) = b_{\lambda}(t),
\end{array}
\end{equation}
where $a,b,a_{\lambda},b_{\lambda}\colon {\mathbb R}^{+}\to {\mathbb R}^{+}$ are smooth functions, for each $\lambda\in \Sigma^{+},$ being zero for the rest of components.

Next, we introduce a family of $G$-invariant almost complex structures on $D^{+}(G/K)$ containing the standard structure $J_{\mid D^{+}(G/K)}.$ Let $q\colon {\mathbb R}^{+}\to {\mathbb R}^{+}$ be a smooth function and let ${\mathcal J}^{q}_{t}$ be the $(1,1)$-tensor on $T_{(o_{H},t)}(G/H\times {\mathbb R}^{+})\cong \overline{\mathfrak m}\times T_{t}{\mathbb R},$ for each $t\in {\mathbb R}^{+},$ given by
\begin{equation}\label{Jq}
\begin{array}{lcl}
{\mathcal J}^{q}_{t}(X,0) = (0,\frac{\partial}{\partial t}), & & {\mathcal J}^{q}_{t}(0,\frac{\partial}{\partial t}) = (-X,0),\\[0.4pc]
{\mathcal J}^{q}_{t}(\xi^{s}_{\lambda},0) = (-\frac{1}{q_{\lambda}(t)}\zeta^{s}_{\lambda},0), & & {\mathcal J}^{q}_{t}(\zeta^{s}_{\lambda},0) = (q_{\lambda}(t)\xi^{s}_{\lambda},0),
\end{array}
\end{equation}
for all $s = 1,\dots, m_{\lambda},$ $\lambda\in \Sigma^{+},$ where $q_{\lambda} = q \circ \lambda_{\mathbb R}.$ Because the subspaces ${\mathfrak m}_{\lambda}$ and ${\mathfrak k}_{\lambda}$ are ${\rm Ad}(H)$-invariant, ${\mathcal J}^{q}$ on $({\mathfrak m}^{+}\oplus {\mathfrak k}^{+})\times\{0\}$ is ${\rm Ad}(H)$-invariant and ${\mathcal J}^{q}$ determines a (unique) $G$-invariant almost complex structure $J^{q}$ on $G/H\times {\mathbb R}^{+}$ such that $J^{q}_{(o_{H},t)} = {\mathcal J}^{q}_{t}.$ 

From Proposition \ref{estandard}, $J^{q},$ $q$ being the identity map, is the standard almost complex structure $J$ on $D^{+}(G/K)$ and the Sasaki metric $g^{S}$ coincides with ${\mathbf g}$ in (\ref{gg}), where the functions $a,$ $b$ and $a_{\lambda}$ are constant, with $a = b = a_{\lambda} = 1,$ and $b_{\lambda} = \lambda^{2}_{\mathbb R}.$

\begin{lemma}\label{lJqX} $J^{q}: = (f^{+})_{*}J^{q}(f^{+})^{-1}_{*}$ on $G\times_{K}({\mathfrak m}\setminus \{0\})$ is given by
$$
\begin{array}{lcl}
J^{q}_{[(a,x)]}\pi_{*}(\xi^{\tt l}_{a},u_{x}) & = & \pi_{*}(-\langle u,{\rm Ad}_{k}X\rangle X^{\tt l}_{ak} - \sum_{\lambda\in \Sigma^{+}} \frac{q_{\lambda}(t)}{\lambda_{\mathbb R}(t)}\sum_{s=1}^{m_{\lambda}}\langle u,{\rm Ad}_{k}\xi^{s}_{\lambda}\rangle(\xi^{s}_{\lambda})^{\tt l}_{ak},\\[0.4pc]
 & & \hspace{1.5cm}\langle \xi,{\rm Ad}_{k}X\rangle X_{tX} + \sum_{\lambda\in \Sigma^{+}}\frac{\lambda_{\mathbb R}(t)}{q_{\lambda}(t)}\sum_{s=1}^{m_{\lambda}}\langle\xi,{\rm Ad}_{k}\xi^{s}_{\lambda}\rangle(\xi^{s}_{\lambda})_{tX}),
\end{array}
$$
for all $a\in G,$ $x = t{\rm Ad}_{k}X,$ $t\in {\mathbb R}^{+}$ and $k\in K,$ and $(\xi,u)\in {\mathfrak m}\times {\mathfrak m}.$
\end{lemma}
\begin{proof} Using (\ref{kk}), (\ref{lGm}) and (\ref{f++}), together with (\ref{Jq}), we have
$$
\begin{array}{lcl}
(J^{q})_{[(a,tX)]}\pi_{*}(X^{\tt l}_{a},0)  =  \pi_{*}(0, X_{tX}), & &\hspace{-3pt} (J^{q})_{[(a,tX)]}\pi_{*}(0,X_{x})  = \pi_{*}(-X^{\tt l}_{a},0),\\[0.4pc]
(J^{q})_{[(a,tX)]}\pi_{*}((\xi^{s}_{\lambda})^{\tt l}_{a},0)  =  \pi_{*}(0,\frac{\lambda_{\mathbb R}(t)}{q_{\lambda}(t)}(\xi^{s}_{\lambda})_{tX}), & & \hspace{-3pt}(J^{q})_{[(a,tX)]}\pi_{*}(0,(\xi^{s}_{\lambda})_{tX})  = \pi_{*}(-\frac{q_{\lambda}(t)}{\lambda_{\mathbb R}(t)}(\xi^{s}_{\lambda})^{\tt l}_{a},0).
\end{array}
$$
Then, for each $(\xi,u)\in {\mathfrak m}\times{\mathfrak m},$
\begin{equation}\label{JqX}
\begin{array}{lcl}
(J^{q})_{[(a,tX)]}\pi_{*}(\xi^{\tt l}_{a},u_{tX}) & = & \pi_{*} (-\langle u,X\rangle X^{\tt l}_{a} - \sum_{\lambda\in \Sigma^{+}} \frac{q_{\lambda}(t)}{\lambda_{\mathbb R}(t)} \sum_{s=1}^{m_{\lambda}}\langle u,\xi^{s}_{\lambda}\rangle(\xi^{s}_{\lambda})^{\tt l}_{a},\\[0.4pc]
& & \hspace{1.5cm}\langle \xi,X\rangle X_{tX}+ \sum_{\lambda\in \Sigma^{+}}\frac{\lambda_{\mathbb R}(t)}{q_{\lambda}(t)}\sum_{s=1}^{m_{\lambda}}\langle\xi,\xi^{s}_{\lambda}\rangle(\xi^{s}_{\lambda})_{tX} ).
\end{array}
\end{equation}

On the other hand, 
$$
\begin{array}{lcl}
\pi_{*}(\xi^{\tt l}_{a},u_{x}) & = & \frac{d}{ds}_{\mid s = 0}\pi(a\exp s\xi, {\rm Ad}_{k}(tX + s{\rm Ad}_{k^{-1}}u))\\[0.4pc]
& = & \frac{d}{ds}_{\mid s = 0} \pi(a(\exp s\xi) k, tX + s{\rm Ad}_{k^{-1}}u) = \pi_{*}(({\rm Ad}_{k^{-1}}\xi)^{\tt l}_{ak},({\rm Ad}_{k^{-1}}u)_{tX}).
 \end{array}
 $$
 Hence, $J^{q}_{[(a,x)]}\pi_{*}(\xi^{\tt l}_{a},u_{x}) = J^{q}_{[(ak,tX)]}\pi_{*}(({\rm Ad}_{k^{-1}}\xi)^{\tt l}_{ak},({\rm Ad}_{k^{-1}}u)_{tX}).$ Now the result follows from (\ref{JqX}), taking into account that $\langle\cdot,\cdot\rangle$ is ${\rm Ad}(K)$-invariant.
\end{proof}

\begin{proposition}\label{TJc} The almost complex structure $J^{q}$ on $D^{+}(G/K)$ can be extended to a $G$-invariant almost complex structure on whole $T(G/K)$ if and only if 
\begin{equation}\label{limit}
0< \lim_{t\to 0^{+}}\frac{q(t)}{t}<+\infty.
\end{equation}
\end{proposition}
 \begin{proof} From Lemma \ref{lJqX}, since ${\rm Ad}(K)({\mathbb R}^{+}) = {\mathfrak m}\setminus\{0\},$ the existence of an extension of $J^{q}$ is determined by the existence with (finite) positive value of $\lim_{t\to 0}\frac{q_{\lambda}(t)}{\lambda_{\mathbb R}(t)} = \lim_{t\to 0}\frac{q(t)}{t}.$
\end{proof}
From (\ref{gg}) and (\ref{Jq}), the following is immediate.
\begin{lemma}\label{lHermitian} The metric ${\mathbf g}$ on $G/H\times {\mathbb R}^{+}$ is Hermitian with respect to $J^{q}$ if and only if $a = b$ and $b_{\lambda} = q^{2}_{\lambda}\cdot a_{\lambda},$ for $\lambda\in \Sigma^{+}.$
\end{lemma}
Moreover, the almost contact metric structure on $G/H\cong G/H\times \{r\},$ induced from the Hermitian structure $(J^{q},{\mathbf g}),$ is the $G$-invariant structure $(\varphi^{q},\frac{1}{a(r)}\xi,a(r)\eta,\tilde{\mathbf g}= (\iota_{r})^{*}{\mathbf g}),$ such that at $o_{H}$ and with respect to the basis $\{X,\xi^{s}_{\lambda},\zeta^{s}_{\lambda};\;s = 1,\dots, m_{\lambda},\lambda\in \Sigma^{+}\}$ of $\overline{\mathfrak m},$ 
 \begin{equation}\label{varphigg} 
 \begin{array}{l}
\varphi^{q}_{o_{H}}\xi^{s}_{\lambda} = -\frac{1}{q_{\lambda}(r)}\zeta^{s}_{\lambda},\quad \varphi^{q}_{o_{H}}\zeta^{s}_{\lambda} = q_{\lambda}(r)\xi^{s}_{\lambda},\quad \tilde{\mathbf g}_{o_{H}} = \tilde{\mathbf g}_{o_{H}}^{(a^{2}(r);a_{\lambda}(r);(q_{\lambda}(r))^{2}a_{\lambda}(r))}.
\end{array}
\end{equation}
\begin{proposition}\label{pcharac} $(\varphi^{q},\frac{1}{a(r)}\xi,a(r)\eta,\tilde{\mathbf g})$ on $T_{r}(G/K),$ for each $r>0,$ is contact metric if and only if
\[
a_{\lambda} (r)= \frac{a(r)\lambda_{\mathbb R}(r)}{2rq_{\lambda}(r)}.
\]
Moreover, it is $K$-contact if and only if $q_{\lambda}(r) = 1,$ for all $\lambda\in \Sigma^{+}.$
 \end{proposition}
\begin{proof} For $u,v\in \overline{\mathfrak m},$ $d\eta_{o_{H}}(u,v) = -\frac{1}{2}\eta_{o_{H}}([u,v]_{\overline{\mathfrak m}}) = -\frac{1}{2}\langle [X,u],v\rangle.$ Then, using (\ref{eq.ms4.3}), 
\[
d\eta_{o_{H}}(\xi^{s}_{\lambda},\zeta^{s}_{\lambda}) = \frac{\lambda_{\mathbb R}(r)}{2r},\quad s=1,\dots ,m_{\lambda},\;\lambda\in \Sigma^{+},
\]
being zero the rest of components of $d\eta_{o_{H}}.$ On the other hand, from (\ref{varphigg}), the fundamental $2$-form $\Phi$ of $(\varphi^{q},\frac{1}{a(r)}\xi,a(r)\eta,\tilde{\mathbf g})$ satisfies $\Phi_{o_{H}}(\xi^{s}_{\lambda},\zeta^{s}_{\lambda}) = q_{\lambda}(r)a_{\lambda}(r).$ Hence, using Proposition \ref{pproperties} (i), the result is proved.
\end{proof}
As a consequence from Propositions \ref{pstandard} and \ref{pcharac}, we have the following version of Tashiro's Theorem \cite{Tash} for tangent sphere bundles of {\em any} radius.
 \begin{corollary}\label{Tashiro} The standard almost contact metric structure $(\varphi,\xi,\eta,\tilde{g}^{S})$ on $T_{r}(G/K)$ is contact metric if and only if $r = \frac{1}{2}.$ Moreover, for all $r>0,$ the rectified almost contact metric structure $(\varphi' = \varphi,\xi' = 2r\xi, \eta' = \frac{1}{2r}\eta,\widetilde{\mathbf g}' = \frac{1}{4r^{2}}\tilde{g}^{S})$ is always contact metric and it is $K$-contact if and only if $r = 1$ and $G/K = {\mathbb S}^{n}$ or ${\mathbb R}{\mathbf P}^{n};$ in that case $(\varphi',\xi',\eta',\widetilde{\mathbf g}')$ is Sasakian.
  \end{corollary}
\begin{remark} {\rm $({\mathbb S}^{n},{\mathbf g})$ and $({\mathbb R}{\mathbf P}^{n},{\mathbf g}),$ described in {\bf Case 1}, have constant curvature $1.$}
\end{remark}

 We conclude with the proof of our main result.
 
 \noindent {\bf Proof of Theorem \ref{main}.} Denote by $J^{1}$ the almost complex structure on $D^{+}(G/K)$ such that $J^{1}_{(o_{H},t)} = {\mathcal J}^{q}_{t},$ for all $t\in {\mathbb R}^{+},$ $q$ being the constant map $q = 1.$ Note that, from Proposition \ref{TJc}, such a structure cannot be extended to an almost complex structure on whole $T(G/K).$ 

For each smooth function $f\colon {\mathbb R}^{+}\to {\mathbb R}^{+},$ consider the compatible $G$-invariant metric ${\mathbf g}^{f}$ to $J^{1},$ such that at $(o_{H},t)\in G/H\times {\mathbb R}^{+}$ it is given as in (\ref{gg}) for $a = b = f$ and $a_{\lambda}(t) = b_{\lambda}(t) =\frac{f(t)\lambda_{\mathbb R}(t)}{2t},$ for all $t\in {\mathbb R}^{+}$ and $\lambda\in \Sigma^{+},$ i.e.
$$
\begin{array}{l}
{\mathbf g}^{f}_{(o_{H},t)}((X,0),(X,0)) = {\mathbf g}^{f}_{(o_{H},t)}((0,\frac{\partial}{\partial t}),(0,\frac{\partial}{\partial t})) = f^{2}(t),\\[0.4pc]
{\mathbf g}^{f}_{(o_{H},t)}((\xi^{s}_{\lambda},0),(\xi^{s}_{\lambda},0)) = {\mathbf g}^{f}_{(o_{H},t)}((\zeta^{s}_{\lambda},0),(\zeta^{s}_{\lambda},0)) = \frac{f(t)\lambda_{\mathbb R}(t)}{2t}, \quad s = 1,\dots, m_{\lambda},\;\lambda\in \Sigma^{+},
\end{array}
$$
 being zero for the rest of components. The induced metric $\tilde{\mathbf g}^{\kappa},$ $\kappa = f(r),$ from ${\mathbf g}^{f}$ on $T_{r}(G/K) = G/H,$ for each $r>0,$ is the $G$-invariant metric satisfying (\ref{fmain}) and, from Pro\-po\-si\-tion \ref{pcharac}, the pair $(\kappa^{-1}\xi,\tilde{\mathbf g}^{\kappa})$ is a $K$-contact structure.
 
 Next, we show that $(\kappa^{-1}\xi,\tilde{\mathbf g}^{\kappa})$ is the unique $G$-invariant $K$-contact structure on $T_{r}(G/K)$ whose characteristic vector field is $\kappa^{-1}\xi.$ Let $\tilde{\mathbf g} = \tilde{\mathbf g}^{(a;a_{\lambda};b_{\lambda})}$ be a $G$-invariant Riemannian metric as in (\ref{Bb}) such that $(\kappa^{-1}\xi,\tilde{\mathbf g})$ is $K$-contact. Then, $a = \kappa$ and, from Proposition \ref{pproperties}, $b_{\lambda} = a_{\lambda},$ for all $\lambda\in \Sigma^{+}.$ Moreover there exists a $G$-invariant $(1,1)$-tensor field $\varphi$ such that $(\varphi,\kappa^{-1}\xi,\kappa\eta,\tilde{\mathbf g})$ is an almost contact metric structure verifying
 \[
 \tilde{\mathbf g}_{o_{H}}(\varphi u,v) = \kappa d\eta_{o_{H}}(u,v) =-\frac{\kappa}{2}\langle[X,u],v\rangle,
 \]
 for all $u,v\in \overline{\mathfrak m}.$ Hence, using (\ref{eq.ms4.3}), $\varphi\xi^{s}_{\lambda} = \frac{\kappa\lambda_{\mathbb R}(r)}{2ra_{\lambda}}\zeta^{s}_{\lambda},$ for all $s = 1,\dots ,m_{\lambda},$ $\lambda\in \Sigma^{+}.$ Then, $a_{\lambda} = \frac{\kappa\lambda_{\mathbb R}(r)}{2r},$ because $\tilde{\mathbf g}$ is a $(\varphi,\kappa^{-1}\xi,\kappa\eta)$-compatible metric. Therefore, $\tilde{\mathbf g} = \tilde{\mathbf g}^{\kappa}$ and $\varphi = \varphi^{1}.$ Applying (\ref{eq.ms4.3}) in (\ref{varphigg}), $\varphi^{1}$ satisfies
  \begin{equation}\label{varphiad}
{\varphi^{1}_{o_{H}}}_{\mid {\mathfrak m}_{\lambda}\oplus {\mathfrak k}_{\lambda}} = \frac{1}{\lambda_{\mathbb R}(X)}{\rm ad}_{X},\quad \lambda\in \Sigma^{+}.
\end{equation}

 Now, we only need to show that $(\varphi^{1},\kappa^{-1}\xi,\kappa\eta,\tilde{\mathbf g}^{\kappa})$ is normal.  By definition (\ref{Nijenhuis}), the $G$-invariant  tensor field ${\mathbf N} = [\varphi^{1},\varphi^{1}] + 2d\eta\otimes\xi$ on $G/H$ is given by 
\begin{equation}\label{NoH}
{\mathbf N}_{o_{H}}(u,v) = -[u,v]_{\overline{\mathfrak m}} + [\varphi^{1}_{o_{H}}u,\varphi^{1}_{o_{H}}v]_{\overline{\mathfrak m}} - \varphi^{1}_{o_{H}}[\varphi^{1}_{o_{H}}u,v]_{\overline{\mathfrak m}} - \varphi^{1}_{o_{H}}[u,\varphi^{1}_{o_{H}}v]_{\overline{\mathfrak m}},
\end{equation}
for all $u,v\in \overline{\mathfrak m}.$ Hence, using that by application of (\ref{varphiad}) we have 
\[{\rm ad}_{X}\circ\varphi^{1}_{o_{H}} = \varphi^{1}_{o_{H}}\circ {\rm ad}_{X},
\]
 we obtain that ${\mathbf N}^{1}_{o_{H}}(X,u) = 0,$ for all $u\in \overline{\mathfrak m},$ and 
\begin{equation}\label{Narphi}
{\mathbf N}_{o_{H}}(u,\varphi^{1}_{o_{H}} v) = {\mathbf N}_{o_{H}}(\varphi^{1}_{o_{H}} u,v),\quad u,v\in \overline{\mathfrak m}.
\end{equation}

If $u,v\in {\mathfrak m}_{\varepsilon}\oplus {\mathfrak k}_{\varepsilon}$ then, from (\ref{brack1}), $[u,v]_{\overline{\mathfrak m}}\in {\mathfrak a}$ and so, $\varphi^{1}_{o_{H}}[u,v]_{\overline{\mathfrak m}} = 0.$ Using (\ref{varphiad}), the Jacobi identity and taking into account that ${\rm ad}_{X}\overline{\mathfrak m}\subset\overline{\mathfrak m},$ it follows that
\begin{equation}\label{sor}
[\varphi^{1}_{o_{H}}u,v]_{\overline{\mathfrak m}} + [u,\varphi^{1}_{o_{H}}v]_{\overline{\mathfrak m}} = 0,\quad [\varphi^{1}_{o_{H}}u, \varphi^{1}_{o_{H}}v]_{\overline{\mathfrak m}} = [u,v]_{\overline{\mathfrak m}}.
\end{equation}
Hence (\ref{NoH}) implies that ${\mathbf N}_{o_{H}}(u,v) =0.$

If $u,v\in {\mathfrak m}_{\varepsilon/2} \oplus {\mathfrak k}_{\varepsilon/2},$ from (\ref{brack1}), $[u,v]_{\overline{\mathfrak m}}\in {\mathfrak a}\oplus{\mathfrak m}_{\varepsilon}\oplus {\mathfrak k}_{\varepsilon}.$ Then, using (\ref{varphiad}) and the Jacobi identity,
\[
\varphi^{1}_{o_{H}}[u,v]_{\overline{\mathfrak m}} = \varphi^{1}_{o_{H}}[u,v]_{{\mathfrak m}_{\varepsilon}\oplus {\mathfrak k}_{\varepsilon}} = \frac{1}{2}([\varphi^{1}_{o_{H}}u,v]_{\overline{\mathfrak m}} + [u,\varphi^{1}_{o_{H}}v]_{\overline{\mathfrak m}}).
\]
Therefore, $[u,v]_{{\mathfrak m}_{\varepsilon} \oplus {\mathfrak k}_{\varepsilon}} = -\frac{1}{2}\varphi^{1}_{o_{H}}([\varphi^{1}_{o_{H}}u,v]_{\overline{\mathfrak m}} + [u,\varphi^{1}_{o_{H}}v]_{\overline{\mathfrak m}})$ and we get
\begin{equation}\label{sorpresa}
[u,v]_{{\mathfrak m}_{\varepsilon} \oplus {\mathfrak k}_{\varepsilon}} = -[\varphi^{1}_{o_{H}}u,\varphi^{1}_{o_{H}}v]_{{\mathfrak m}{\varepsilon} \oplus {\mathfrak k}_{\varepsilon}},\quad [\varphi^{1}_{o_{H}}u,v]_{{\mathfrak m}_{\varepsilon}\oplus {\mathfrak k}_{\varepsilon}}= [u, \varphi^{1}_{o_{H}}v]_{{\mathfrak m}_{\varepsilon}\oplus {\mathfrak k}_{\varepsilon}}.
\end{equation}
Because, applying (\ref{varphiad}) and (\ref{NoH}), 
\[
\langle{\mathbf N}_{o_{H}}(u,v),X\rangle = \langle u,{\rm ad}_{X}v\rangle - \langle\varphi^{1}_{o_{H}}u,{\rm ad}_{X}\varphi^{1}_{o_{H}}v\rangle = 0,
\]
 we have ${\mathbf N}(u,v) = {\mathbf N}(u,v)_{\mid{\mathfrak m}_{\varepsilon}\oplus {\mathfrak k}_{\varepsilon}} = 0.$

 Finally, consider the case $u\in {\mathfrak m}_{\varepsilon}\oplus {\mathfrak k}_{\varepsilon}$ and $v\in{\mathfrak m}_{\varepsilon/2}\oplus {\mathfrak k}_{\varepsilon/2}$ or, equivalently using (\ref{Narphi}), the case $u\in {\mathfrak m}_{\varepsilon}\oplus {\mathfrak k}_{\varepsilon}$ and $v\in{\mathfrak m}_{\varepsilon/2}.$ From (\ref{brack1}), $[u,v]\in {\mathfrak k}_{\varepsilon/2}\oplus {\mathfrak m}_{\varepsilon/2}$ and then, 
 \[
 \varphi^{1}_{o_{H}}[u,v] = 2{\rm ad}_{X}[u,v].
 \]
 Applying again the Jacobi identity, we obtain $[u,v] = -\varphi^{1}_{o_{H}}(2[\varphi^{1}_{o_{H}}u,v] + [u,\varphi^{1}_{o_{H}}v]).$ Hence, (\ref{NoH}) can be expressed as 
 \[
 {\mathbf N}_{o_{H}}(u,v) = \varphi^{1}_{o_{H}}([\varphi^{1}_{o_{H}}u,v] + [u,\varphi^{1}_{o_{H}}v]).\]
 Therefore, we have ${\mathbf N}_{o_{H}}(u,v) = 0$ if and only if 
\begin{equation}\label{S1}
[\varphi^{1}_{o_{H}}u,\varphi^{1}_{o_{H}}v] = [u,v],\quad\mbox{for all}\; u\in {\mathfrak m}_{\varepsilon}\oplus{\mathfrak k}_{\varepsilon},\; v\in {\mathfrak m}_{\varepsilon/2}.
\end{equation}
From (\ref{varphigg}), equality (\ref{S1}) is equivalent to (\ref{S2}) and the result follows from Lemma \ref{pbrack}.

\end{document}